\documentclass[11pt,reqno]{amsart}
\usepackage[pagewise]{lineno}
\usepackage{txfonts}
\usepackage{amsmath,amsthm,amsfonts}
\usepackage{latexsym}
\usepackage{amssymb}
\usepackage{mathrsfs}
\usepackage{appendix}
\usepackage[colorlinks,
            linkcolor=red,
            anchorcolor=blue,
            citecolor=green,
            ]{hyperref}
\usepackage[dvips]{epsfig}

\newcounter{RomanNumber}

\renewcommand{\thesection}{\arabic{section}}

\newtheorem{theorem}{Theorem}[]
\newtheorem{lemma}{Lemma}[section]
\newtheorem{prop}{Proposition}[section]
\newtheorem{defi}{Definition}[section]

\theoremstyle{remark}
\newtheorem{remark}{Remark}[section]

\renewcommand{\theequation}{\thesection .\arabic{equation}}
\let\sect\section
\renewcommand\section{\setcounter{equation}{0}
\gdef\theequation{\thesection .\arabic{equation}}\sect}

\makeatletter

\newcommand{\Rmnum}[1]{\expandafter\@slowromancap\romannumeral #1@}
\makeatother

\newcommand{\cA}{{\mathcal{A}}}
\newcommand{\cB}{{\mathcal{B}}}
\newcommand{\cD}{{\mathcal{D}}}
\newcommand{\cE}{{\mathcal{E}}}
\newcommand{\cG}{{\mathcal{G}}}

\newcommand{\cN}{{\mathcal{N}}}

\newcommand{\cM}{{\mathcal{M}}}

\newcommand{\cS}{{\mathcal{S}}}

\newcommand{\cP}{{\mathcal{P}}}

\newcommand{\IC}{{\mathbb{C}}}

\newcommand{\IR}{{\mathbb{R}}}
\newcommand{\TT}{{\mathbb{T}}}

\newcommand{\be}{\begin{eqnarray}}
\newcommand{\ee}{\end{eqnarray}}

\newcommand{\mes}{\mathop{\rm{mes}\, }}

\def\beeq{\begin{equation}}
\def\eneq{\end{equation}}

\def\bm{\begin{matrix}}
\def\endm{\end{matrix}}
\def\cS{{\mathcal S}}

\def\Im{{\rm Im}}

\begin{document}

\title[LDT for   Dirichlet determinants with Brjuno-R\"ussmann frequency ]{Large Deviation theorems for  Dirichlet determinants of analytic quasi-periodic  Jacobi  operators with Brjuno-R\"ussmann frequency }

\author{Wenmeng Geng}
	\address{College of Sciences, Hohai University, 1 Xikang Road Nanjing Jiangsu 210098 P.R.China}
	\email{wenmeng$\underline{\ }$geng@163.com}

\author{Kai Tao}
	\address{College of Sciences, Hohai University, 1 Xikang Road Nanjing Jiangsu 210098 P.R.China}
	\email{ktao@hhu.edu.cn,\ tao.nju@gmail.com}


\thanks{The second author was supported by China Postdoctoral Science Foundation (Grant 2019M650094), the National Nature Science Foundation of China (Grant 11401166) and the Fundamental Research Funds for the Central Universities. He is the corresponding author.}

\date{}

\begin{abstract}
In this paper, we first study  the strong Birkhoff Ergodic Theorem for subharmonic functions with the Brjuno-R\"ussmann shift on the Torus. Then, we apply it to prove the large deviation theorems for   the
finite scale Dirichlet determinants of quasi-periodic analytic Jacobi  operators with this frequency. It shows that  the Brjuno-R\"ussmann function, which reflects the irrationality of the frequency, plays the key role in these theorems via the smallest deviation. At last, as an application, we obtain  a distribution of the eigenvalues of the Jacobi operators with Dirichlet boundary conditions, which also depends on the smallest deviation, essentially on the irrationality of the frequency.
\end{abstract}

\maketitle

\section{Introduction}
We study the following quasi-periodic analytic Jacobi operators $H(x,\omega)$ on
$l^2(\mathbb{Z})$:
\begin{equation}\label{jacobiequ}
\left[H(x,\omega)\phi\right](n)=-a(x+(n+1)\omega)\phi(n+1)-\overline{a(x+n\omega)}\phi(n-1)+v(x+n\omega)\phi(n)
,\ n\in\mathbb{Z},\end{equation}where $v:\mathbb{T}\to\mathbb{R}$ is a real analytic function called potential, $a:\mathbb{T}\to\mathbb{C}$ is a complex analytic function and not identically zero.  The characteristic equations  $H(x,\omega)\phi=E\phi$ can be expressed as
\begin{equation}
\left (\begin{array}{cc}
  \phi(n+1) \\ \phi(n) \\
\end{array} \right )=\frac{1}{a(x+(n+1)\omega)}\left ( \begin{array}{cc}
  v(x+n\omega)-E & -\overline{a(x+n\omega)} \\
 a(x+(n+1)\omega)& 0 \\
  \end{array}\right )\left (\begin{array}{cc}
  \phi(n) \\ \phi(n-1) \\
\end{array} \right ).
\end{equation}Define
\begin{equation}\label{eq:1.3}
  M(x,E,\omega):=\frac{1}{a(x+\omega)}\left ( \begin{array}{cc}
 v(x)-E & -\overline{a(x)} \\
 a(x+\omega)& 0 \\
  \end{array}\right )
\end{equation}and  call a map
$$(\omega,M):(x,\vec v)\mapsto (x+\omega,M(x)\vec v)$$ a Jacobi cocycle. Due to the fact that an analytic function only has  finite zeros, $M(x,E,\omega)$ and  the n-step transfer matrix $$M_n(x,E,\omega):=\prod_{k=n}^1M(x+k\omega,E)$$ make sense almost everywhere.
By the Kingman's subadditive ergodic theorem, the Lyapunov exponent
\begin{equation}\label{leomega}
  L(E,\omega)=\lim\limits_{n\to\infty}L_n(E,\omega)=\inf\limits_{n\to\infty}L_n(E,\omega)\ge 0
\end{equation} always exists, where
\[L_n(E,\omega)=\frac{1}{n}\int_{\mathbb{T}}\log\|M_n(x,E,\omega)\|dx.\]
Let $H_{[m, n]}(x, \omega)$ be the Jaocbi operator defined by (\ref{jacobiequ}) on a finite interval $[m,n]$ with Dirichlet boundary conditions, $\phi(m-1)=0$ and  $\phi(n+1)=0$. Let $f^a_{[m, n]}(x, E,\omega)=\det (H_{[m, n]}(x, \omega)-E)$ be its characteristic polynomial.
One has
\begin{equation}\label{1.Dirdet'}
f^a_{[m,n]}(x,  E,\omega) = f^a_{n-m+1} \bigl(x+(m-1)\omega,  E,\omega \bigr),
\end{equation}
where
\begin{equation}\label{Dirichlet-det}
    \begin{aligned}
    f^a_n(x, E,\omega)  &= \det\bigl(H_n(x, \omega)-E\bigr)\\
    & =
    \begin{vmatrix}
     v\bigl(x+\omega\bigr) - E & -a\bigl(x+2\omega\bigr) & 0  &\cdots &\cdots & 0\\[5pt]
    -\overline{a\bigl(x+2\omega\bigr)} & v\bigl(x+2\omega\bigr) - E & -a\bigl(x+3\omega\bigr) & 0 & \cdots & 0\\[5pt]
    \vdots & \vdots & \vdots & \vdots & \vdots & \vdots\\
    &&&&&-a\bigl(x+n\omega\bigr) \\[5pt]
    0 & \dotfill & 0 &0 &-\overline{a\bigl(x+n\omega\bigr)} & v\bigl(x+n\omega\bigr) - E
    \end{vmatrix}
    \end{aligned}.
\end{equation}

In this paper, the aim is to study the properties of $f_N^a(x,E,\omega)$. To state our conclusions, we first make some introductions to the background of our topic.

The operator (\ref{jacobiequ}) has  the following important special case, which is called the Schr\"odinger operator and  has been studied extensively:
\begin{equation}\label{sch-equ}
\left[H^s(x,\omega)\phi\right](n)=\phi(n+1)+\phi(n-1)+v(x+n\omega)\phi(n)
,\ n\in\mathbb{Z}.\end{equation}
Then, $M_n^s(x,E,\omega)$, $L^s(E,\omega)$, $L^s_n(E,\omega)$ and $f^s_n(x,E,\omega)$ have the similar definitions.
In \cite{BG00}, Bourgain and Goldstein proved that if $L^s(E,\omega)>0$, then for almost all $\omega$, the operator $H^s(0,\omega)$ has Anderson Localization, which means that it has pure-point spectrum with exponentially decaying eigenfunction. In \cite{GS01}, Goldstein and Schlag obtained the H\"older continuity of $L^s(E,\omega)$ in $E$ with the strong Diophantine $\omega$, i.e. for some $\alpha>1$ and any integer $n$,
\begin{equation}
  \label{sdc}
  \|n\omega\|>\frac{C_{\omega}}{|n|\left(\log|n|+1\right)^{\alpha}}.
\end{equation}
It is well known that for a fixed $\alpha>1$, almost every irrational $\omega$ satisfies (\ref{sdc}). Obviously, if we define the Diophantine number as
\begin{equation}
  \label{dc}
  \|n\omega\|>\frac{C_{\omega}}{|n|^{\alpha}},
\end{equation}
then it also has a full measure. In these two references, the key lemmas are the following called  large deviation theorems (LDTs for short) for matrix $M_n^s(x,E,\omega)$ with these two frequencies: for the Diophantine $\omega$, it was proved in \cite{BG00} that  there exists $0<\sigma<1$ such that
 \begin{equation}
   \label{ldtinbg00}
 \mes\{x:\left|\frac{1}{n}\log\|M^s_n(x,E,\omega)\|-L^s_n(E,\omega)\right|> n^{-\sigma}\}<\exp\left(-  n^{\sigma}\right);
 \end{equation}
for the strong Diophantine $\omega$, it was proved in \cite{GS01} that there exists $\delta_0^n=\frac{(\log n)^A}{n}$ such that for any $\delta>\delta_0^n$
\begin{equation}
   \label{ldtings01}
 \mes\{x:\left|\frac{1}{n}\log\|M^s_n(x,E,\omega)\|-L^s_n(E,\omega)\right|> \delta\}<\exp\left(-  c\delta^2 n\right).
 \end{equation}
 Here $\delta_0^n$ is called the smallest deviation in the LDT and  very important in our paper.

Compared with the Schr\"odinger cocycle, one of  the  distinguishing features of the Jacobi cocycle is that it is not $SL(2,\mathbb{C})$. Then Jitomirskaya, Koslover and Schulteis \cite{JKS09}, and Jitomirskaya and Marx \cite{JM11} proved that the LDT (\ref{ldtinbg00}) for $M_n(x,E,\omega)$ and the weak H\"older continuity of the Lyapunov exponent of  the analytic $GL(2,\mathbb{C})$ cocycles  hold with the  Diophantine frequency. In \cite{T14}, we showed that (\ref{ldtings01}) can hold for $M_n(x,E,\omega)$ with the strong Diophantine $\omega$ and the continuity of the Lyapunov exponent of the Jacobi cocycles $L(E,\omega)$  can be H\"older in $E$.

For any irrational $\omega$, there exist its continued fraction approximates $\{\frac{p_s}{q_s}\}_{s=1}^{\infty}$,  satisfying
\begin{equation}\label{irtor0}
 \frac{1}{q_s(q_{s+1}+q_s)}<|\omega-\frac{ p_s}{q_s}|<\frac{1}{q_sq_{s+1}}.
\end{equation}
Define $\beta$ as the exponential growth exponent of $\{\frac{p_s}{q_s}\}_{s=1}^{\infty}$ as follows:
\[
 \beta(\omega):=\limsup_s \frac{\log q_{s+1}}{q_s}\in[0,\infty].
\]
Obviously, if $\omega$ is strong Diophantine or Diophantine, then $\beta(\omega)=0$. We say $\omega$ is the Liouville number, if $\beta(\omega)>0$.
Recently, more and more attentions are paid to the question that what will happen to  these operators with more generic $\omega$. So far, the most striking answers are mainly  for the almost Mathieu operators (AMO for short), which is also a special case of the Jacobi ones
\begin{equation}\label{am-equ}
\left[H^m(x,\omega,\lambda)\phi\right](n)=\phi(n+1)+\phi(n-1)+2\lambda\cos\left(2\pi(x+n\omega)\right)\phi(n)
.\ n\in\mathbb{Z}.\end{equation}
The most famous one, the Ten Martini Problem,  which was dubbed  by Barry Simon and  conjectures that for any irrational $\omega$, the spectrum of AMO is a Cantor set, was completely solved by
Avila and Jitomirskaya \cite{AJ09}. In that reference, they also proved that $H^{m}(x,\omega,\lambda)$ has Anderson Localization for almost every $x\in\mathbb{T}$ with $\lambda>e^{\frac{16}{9}\beta}$. In \cite{AYZ17}, Avila, You and Zhou improved it to $\lambda>e^{\beta}$.

While, the answers for the Schr\"odinger or Jacobi operators  in the positive Lyapunov exponent regimes are mainly in the study of the continuity of the Lyapunov exponent. In \cite{BJ02}, they proved that the Lyapunov exponent is continuous in $E$ for any irrational $\omega$. The first result that the H\"older continuity holds for some weak Liouville frequency, which means that $\beta(\omega)<c$, where $c$ is a small constant depending only on the analytic potential $v(x)$, is \cite{YZ14}. Recently, Han and Zhang \cite{HZ18} ameliorated it to $\lambda>e^{C\beta}$ in the large coupling regimes, where the potential $v$ is of the form $\lambda v_0$ with a general analytic $v_0$ and $C$ is a positive constant also depending only on $v_0$. Our second author also proved the corresponding conclusion for the Jacobi operators in \cite{T18}. These two results are optimal, because Avila, Last, Shamis and Zhou \cite{ALSZ} showed that the continuity of the Lyapunov exponent of the almost Mathieu operators can't be H\"older  if $\beta>0$ and $e^{-\beta}<\lambda<e^{\beta}$.

Until now, we do not know much about the spectrum of the Schr\"odinger or Jacobi operators in the positive Lyapunov exponent regimes when the frequency is not strong Diophantine.
The main reason is that we do not know much about the finite-volume determinant $f^a_n(x,E,\omega)$. While, for the almost Mathieu operators, it can be handled explicitly via the Lagrange interpolation for the trigonometric polynomial. This method can  be applied for the following extend Harper's operators, which also have the cosine potential, to obtain many spectral conclusions with the generic frequency, such as \cite{AJM17} and \cite{H18}:
\begin{eqnarray*}
  a(x)&=&\lambda_3\exp[-2\pi i(x+\frac{\omega}{2})]+\lambda_2+\lambda_1\exp[2\pi i(x+\frac{\omega}{2})],\ \ 0\leq \lambda_2, 0\leq \lambda_1+\lambda_3,\\
  v(x)&=&2\cos(2\pi x).
\end{eqnarray*}
However, the Lagrange interpolation can not work for the Schr\"odinger or Jacobi operators, since their potentials both are generic analytic functions. Therefore,  in \cite{GS08},  Goldstein and Schlag applied the LDT (\ref{ldtings01}) and the relationship that
\begin{equation}
  \label{msn-fsn}
  M_n^s(x,E,\omega)=\left(\begin{array}
    {cc}f_n^s(x,E,\omega)&f^s_{n-1}(x+\omega,E,\omega)\\
    f^s_{n-1}(x,E,\omega)&f^s_{n-2}(x+\omega,E,\omega)
  \end{array}\right)
\end{equation}
 to estimate the BMO norm of $f^s_n(x,E,\omega)$. Then they obtained the following LDT for $f^s_n(x,E,\omega)$ with the strong Diophantine $\omega$ by the John-Nirenberg inequality:
\begin{equation}\label{ldtings08}
 \mes\left\{ x\in\mathbb{T}:\,\left|\log\left|f^s_{n}\left(x\right)\right|-\left\langle \log\left|f^s_{n}\right|\right\rangle \right|>n\delta\right\} \le C\exp\left(-c\delta n(\delta_0^n)^{-1}\right).
\end{equation}
This LDT was applied  to get the H\"older exponent of the H\"older continuity of $L^s(E,\omega)$ in $E$ and  the upper bound on the number of eigenvalues of $H^s_n(x,\omega)$ contained in an interval of size $n^{-C}$. What's more,  with its help, the estimation on the separation of the eigenvalues of $H^s_n(x,\omega)$ and the property that the spectrum of $H^s(x,\omega)$, denoted by $\mathcal{S}_{\omega}$,  is a Cantor set were obtained in \cite{GS11}, and the homogeneity of $\mathcal{S}_{\omega}$ was proved in \cite{GDSV}. In \cite{BV13} and \cite{BV14}, Binder and Voda applied this method to our analytic Jacobi operators (\ref{jacobiequ}). It must be noted that the above conclusions all hold only for the strong Diophantine $\omega$ and  the LDTs for $f^s_n(x,E,\omega)$ and $f^a_n(x,E,\omega)$ are the key lemma in the method.

Now, we can declare that the concrete content of our main aim is to obtain the LDT for the finite-volume determinant $f^a_n(x,E,\omega)$ with more generic $\omega$. It is the preparation for the study of the spectrum problem for discrete quasiperiodic operators of second order in the future.

In this paper, we assume that the frequency $\omega$ is the Brjuno-R\"ussmann number, which is a famous extension of the strong Diophantine number. It says that there exists a monotone increasing and continuous function $\Delta(t):[1,\infty)\to[1,\infty)$ such that $\Delta(1)=1$ and for any positive integer $k>0$,
\begin{equation}
  \label{brn}
  \|k\omega\|>\frac{C_{\omega}}{\Delta(k)},
\end{equation}
and
\begin{equation}
\label{br}
\int_1^{\infty}\frac{\log \Delta(t)}{t^2}<+\infty.
\end{equation}
For example, this Brjuno-R\"ussmann function $\Delta(t)$ can be $t(\log t+1)^{\alpha}$, $t^{\alpha}$, $\exp\left((\log t)^{\alpha}\right)$,  $\exp\left(t^{\frac{1}{\alpha}}\right)$ and $\exp\left(\frac{t}{(\log t)^{\alpha}}\right)$ with $\alpha>1$.
Define $\Gamma_{\omega}(n)=\|n\omega\|^{-1}$. Due to (\ref{irtor0}), we have that
\begin{equation}
  q_{s+1}<\Gamma_{\omega}(q_s)<q_s+q_{s+1},\ \mbox{and}\ \Gamma_{\omega}(n)<\Gamma_{\omega}(q_s),\ \forall n\in (q_s,q_{s+1}).
\end{equation}
Therefore, there exists another definition of the Brjuno-R\"ussmann number as follow: There exists a function $\Psi_{\omega}(t)=\max\{\|k\omega\|^{-1},\ \forall 0<k\le t, k\in \mathbb{Z}\}$  satisfying (\ref{br}).
Note that the denominator series $\{q_s\}_{s=1}^{\infty}$ and the function $\Psi_{\omega}$ depend on $\omega$. Thus, to make almost every irrational number satisfy (\ref{brn}), we assume that
\begin{equation}
  \label{brsd}
  \Delta(t)>t(\log t+1).
\end{equation}
It implies that $\log \Delta(t)>\log t$ but  it is false that $\frac{\Delta'(t)}{\Delta(t)}=\left(\log \Delta(t)\right)'>\left(\log t\right)'=\frac{1}{t}$.
However, there always exists another function $\tilde{\Delta}(t)$ which is larger than and close to $\Delta(t)$, and satisfies (\ref{br}) and that $t\tilde \Delta'(t)\ge \tilde \Delta(t)$ for any $t\ge 1$. So it is very reasonable for us to make the following hypothesis:~\\
\noindent {\bf Hypothesis H.1} $ \Delta(t)>t(\log t+1)$ and $t\Delta'(t)\ge \Delta (t)$ for any $t\ge 1$.~\\
\noindent
Then, our first LDT for $f^a_n(x,E,\omega)$ is
\begin{theorem}\label{ldtforfan}
Let $\omega$ be the Brjuno-R\"ussmann number satisfying Hypothesis H.1 and $L(E,\omega)>0$. There exist  constants $c=c(a,v,E,\omega)$ and $\breve C=\breve C(a,v,\omega)$, and absolute  constant $C$ such that for any integer $n\ge 1$ and $\delta>\delta_{H.1}(n):=\frac{\breve C\log\Delta(n)}{\left(\Delta^{-1}(C_{\omega}n)\right)^{1-}}$,
\[
\mes\left\{ x\in\mathbb{T}:\,\left|\log\left|f^a_{n}\left(x\right)\right|-\left\langle \log\left|f^a_{n}\right|\right\rangle \right|>n\delta\right\} \le C\exp\left(-c\delta (\delta_{H.1}(n))^{-1}\right).
\]
\end{theorem}
\begin{remark}
  In this paper, the notation $A^{1-}$ means $A^{1-\epsilon}$ for any small absolute $\epsilon>0$. And $A^{1+}$ has the similar definition.
\end{remark}
\begin{remark}
 If we assume the potential $v$ is of the form $\lambda v_0$ with a general analytic $v_0$, then the second author \cite{T18} proved that there exists $\lambda_0=\lambda_0(v_0,a)$ such that the Lyapunov exponent $L(E,\omega)$ is always positive for any $E$ and any irrational $\omega$ under the condition $\lambda>\lambda_0$.
\end{remark}

If $\Delta(t)=t(\log t+1)^{\alpha}$, then $\delta_{H.1}(n)=\frac{\breve C\log n}{n^{1-}}$ which is very close to the smallest deviation for the strong Diophantine number $\frac{(\log n)^A}{n}$. But if $\Delta(t)=\exp\left(t^{\frac{1}{\alpha}}\right)$, then $\delta_{H.1}(n)=\breve Cn^{\frac{1}{\alpha}-}$ which is  too large for us to apply Theorem \ref{ldtforfan} to the research of the spectrum of the analytic quasi-periodic operators (\ref{jacobiequ}) in our future work. Thus, we need to make some hypothesis to improve this smallest deviation when $\Delta(t)$ grows fast:~\\
\noindent {\bf Hypothesis H.2} $\ $ $\omega$ satisfies Hypothesis H.1 and for any $t\ge 1$,
\begin{equation}
  \label{uppbound}
  \Delta(t)<\exp\left(\frac{t}{\log t}\right).
\end{equation}
From (\ref{br}), $\Delta(t)$ has an upper bound of $\exp\left(\frac{t}{\log t}\right)$ generally. But it is possible that  it  grows very fast and exceeds this upper bound in some intervals, and in the rest  it grows very slowly and makes the integral converge. Therefore, the aim of this hypothesis  is to avoid this extreme possibility. Then, our second LDT for $f^a_n(x,E,\omega)$ is
\begin{theorem}\label{ldth1}
Let $\omega$ be the Brjuno-R\"ussmann number satisfying Hypothesis H.2 and $L(E,\omega)>0$. There exist   constants $c=c(a,v,E,\omega)$ and $\breve C=\breve C(a,v,\omega)$, and absolute  constant $C$  such that for any integer $n\ge 1$ and $\delta>\delta_{H.2}(n):=\frac{\breve C}{\left[\log(\Delta^{-1}(C_\omega n))\right]^{1-}}$,
\[
\mes\left\{ x\in\mathbb{T}:\,\left|\log\left|f^a_{n}\left(x\right)\right|-\left\langle \log\left|f^a_{n}\right|\right\rangle \right|>n\delta\right\} \le C\exp\left(-c\delta (\delta_{H.2}(n))^{-1}\right).
\]
\end{theorem}
\begin{remark}
  With this hypothesis, no matter $\Delta(t)$ equals to $\exp\left(t^{\frac{1}{\alpha}}\right)$ or $\exp\left(\frac{t}{(\log t)^{\alpha}}\right)$, the smallest deviation $\delta_{H.2}(n)=\frac{\breve C}{\left[\log\log n\right]^{1-}}\ll 1$ which satisfies what we need for the study of the spectrum, such as Theorem \ref{thm:Number-of-eigenvalues}.
\end{remark}

As mentioned above, what we want to avoid is the case that $\Delta(t)$ grows faster than $\exp(t)$ in some intervals. But Hypothesis H.2 only requires that $\Delta(t)$ has an upper bound, but has no restriction on its derivative. Thus, we make the following hypothesis, which gives the mutual restriction between $\Delta(t)$ and $\Delta'(t)$ and looks also very reasonable:~\\
\noindent {\bf Hypothesis H.3} $\ $ $\omega$ satisfies Hypothesis H.1 and $\frac{\log \Delta(t)}{t}$ is non-increasing for any $t\ge 1$.~\\
Easy computation shows that this hypothesis is equivalent to the inequality
\[\Delta'(t)\le \frac{\Delta(t)\log \Delta(t)}{t}.\]
Combined it with Hypothesis H.1, it shows that the  bound of $t\Delta'(t)$ is determined by  $\Delta(t)$. Obviously, all examples of  functions mentioned above satisfy this hypothesis. With its help, we improve Theorem \ref{ldtforfan} and  \ref{ldth1} as follows:
\begin{theorem}\label{ldth2}
Let $\omega$ be the Brjuno-R\"ussmann number satisfying Hypothesis H.3 and $L(E,\omega)>0$. There exist   constants $c=c(a,v,E,\omega)$ and $\breve C=\breve C(a,v,\omega)$, and absolute  constant $C$  such that for any integer $n\ge 1$ and $\delta>\delta_{H.3}(n):=\frac{\breve C\log ( C_{\omega}n)}{\left[\Delta^{-1}(C_{\omega}n)\right]^{1-}}$,
\[
\mes\left\{ x\in\mathbb{T}:\,\left|\log\left|f^a_{n}\left(x\right)\right|-\left\langle \log\left|f^a_{n}\right|\right\rangle \right|>n\delta\right\} \le C\exp\left(-c\delta (\delta_{H.3}(n))^{-1}\right).
\]
\end{theorem}
\begin{remark}
 Since $t(\log t+1)<\Delta(t)$,  it is obvious that $$\delta_{H.1}(n)=\frac{\log\Delta(n)}{\left(\Delta^{-1}(C_{\omega}n)\right)^{1-}}\gg\delta_{H.3}(n)=\frac{\log ( C_{\omega}n)}{\left[\Delta^{-1}(C_{\omega}n)\right]^{1-}}.$$
 On the other hand, due to the fact that $0<\Delta^{-1}(t)<t<\Delta(t)$, $n\gg \Delta\left(\log^2\Delta^{-1}(n)\right)$, we have that
 \[\delta_{H.2}(n)=\frac{1}{\left[\log(\Delta^{-1}(C_\omega n))\right]^{1-}}\gg \delta_{H.3}(n)=\frac{\log ( C_{\omega}n)}{\left[\Delta^{-1}(C_{\omega}n)\right]^{1-}}.\]
\end{remark}~\\

 The key to prove these three LDTs for $f_n^a(x,E,\omega)$ is  an ergodic theorem for the subharmonic function shifting on $\mathbb{T}$. Specifically, we know that if $T:X \to X$ is an ergodic transformation on a measurable space $(X,\Sigma,m)$  and $f$ is an $m-$integrable  function, then the Birkhoff Ergodic Theorem tells that the time average functions $f_n(x)=\frac{1}{n}\sum_{k=0}^{n-1}f(T^kx)$ converge to the space average $\langle f\rangle=\frac{1}{m(X)}\int_X fdm$ for almost every $x\in X$. But it doesn't tell us how fast do they converge? So, we call a  theorem the strong Birkhoff Ergodic Theorem, if it gives the convergence rate. The following  strong Birkhoff Ergodic Theorem for the subharmonic function shifting on $\mathbb{T}$ is the key which we just mentioned above:
 \begin{theorem}
  \label{sbet}
 Let $u:\Omega\to \IR$ be a subharmonic function on
a domain $\Omega\subset\IC$ and $\omega$ be the Brjuno-R\"ussman number satisfying Hypothesis H.1, or H.2, or H.3. Suppose that $\partial \Omega$ consists
of finitely many piece-wise $C^1$ curves and  $\mathbb{T}$ is contained in $\Omega'\Subset \Omega$(i.e., $\Omega'$ is a compactly contained subregion of~$\Omega$). There exist constants  $c=c(\omega, u)$ and $\breve C=\breve C(\Omega, u)$ such that for any positive $n$ and $\delta>\delta_0^n$,
\begin{equation}
  \label{eq-sbet}
  \mes\left (\left \{x\in \mathbb{T}:|\sum_{k=1}^n u(x+k\omega)-n\langle u\rangle|>\delta n\right \}\right ) \leq  \exp\left(-c\delta n\right),
\end{equation}
where
\begin{equation}
  \label{set-delta0n}
  \delta_0^n=\left\{\begin{array}
    {cccc}
    \delta_{H.1}(n):=&\frac{\breve C\log\Delta(n)}{\left(\Delta^{-1}(C_{\omega}n)\right)^{1-}},&\ &\mbox{if}\ \omega\ \mbox{satisfies H.1},\\
    \delta_{H.2}(n):=&\frac{\breve C}{\left[\log(\Delta^{-1}(C_\omega n))\right]^{1-}},&\ &\mbox{if}\ \omega\ \mbox{satisfies H.2},\\
    \delta_{H.3}(n):=&\frac{\breve C\log ( C_{\omega}n)}{\left[\Delta^{-1}(C_{\omega}n)\right]^{1-}} ,&\ &\mbox{if}\ \omega\ \mbox{satisfies H.3}.
  \end{array}\right.
\end{equation}
 \end{theorem}
  A very interesting thing we find is that  no matter the irrational frequency is, the convergence rate of the exceptional measure  is always $\exp\left(-c\delta n\right)$. The only difference is the smallest deviation $\delta_0^n$. If $\beta(\omega)>0$, then our second author obtained in \cite{T18} that $\delta_0^n=c\beta$ which is proved to be optimal in \cite{ALSZ}; if $\beta(\omega)=0$, we obtain (\ref{set-delta0n}) which includes the result for the strong Diophantine number by Goldstein and Schlag. Correspondingly, the three LDTs we obtain in this paper can be unified into the following form:
  \begin{equation}
    \label{ldtforfan-gen}
   \mes\left\{ x\in\mathbb{T}:\,\left|\log\left|f^a_{n}\left(x\right)\right|-\left\langle \log\left|f^a_{n}\right|\right\rangle \right|>n\delta\right\} \le C\exp\left(-c\delta (\delta_{0}^n)^{-1}\right),\ \forall \delta>\delta_0^n.
  \end{equation}
   While, the exceptional measure in (\ref{ldtforfan-gen}) will not converge when $\delta_0^n=c\beta$! The method created by Goldstein and Schlag and applied in this paper should be improved for the Liouville frequency. We think it is a good question for our further research in the future.

   Here we need to emphasize that our paper is not weaker version of \cite{T18}. That shows that the smallest deviation is $\delta_0^n=c\beta$, and then the strong Birkhoff ergodic theorem and the LDTs for matrices hold when the deviation is larger than $\delta_0^n$. Letting the positive Lyapunov exponent be this deviation, our second author obtained the H\"older continuity of the Lyapunov exponent. However, if we applied these results in our condition that $\beta=0$, then the smallest deviation is $0$! It is absurd! So, compared to \cite{T18}, the main aim of our second section is to find the smallest deviation when $\beta=0$. What's more, we will find that in technology the key  is to estimate $\sum_{j=1}^{2m-1}\frac{q_{s-j+1}}{q_{s-j}}\log q_{s-j+1}$. It is  easy when $\beta>0$:
   \[ \sum_{j=1}^{2m-1}\frac{q_{s-j+1}}{q_{s-j}}\log q_{s-j+1}\le 2\beta \sum_{j=1}^{2m-1}q_{s-j+1}\le 2\beta n.\]
   While, when $\beta=0$, the fact that $\left\{\frac{\log q_{s+1}}{q_s}\right\}_{s=1}^{\infty}$ has different speeds, which depend on $\Delta(t)$,  to converge to $0$ makes this estimation much harder. On the other hand, the aims of our Section 3 and 4 are to obtain the LDT for $f_n^a$ and its applications, which are nonexistent in \cite{T18}. In summary, the focus point of our paper is to show the importance of the smallest deviation of the strong Birkhoff ergodic theorem and calculate it when $\beta=0$. Of course, when we need the LDTs for matrices and the H\"older continuity of the Lyapunov exponent, such as Lemma \ref{ldtformatrices} and \ref{holder-contin}, we can use the results from \cite{T18} directly.~\\

At last, we have an application of our LDTs, which estimates the upper bound on the number of eigenvalues of $H_n(x,\omega)$ contained in an interval of size $\left(\delta_0^n\right)^{\frac{1}{h}}$, where $h$ is the H\"older exponent of the H\"older continuity of $L(E,\omega)$, see Lemma \ref{holder-contin}.  The distribution of the eigenvalues is very important in the further study of the spectrum problem for discrete quasiperiodic operators of second order. With fixed $x$ and $\omega$, the matrix $H_n(x,\omega)$ has $n$ eigenvalues. So we have an intuition that these eigenvalues have a more uniform distribution when the frequency $\omega$ is ``more irrational". For the Brjuno-R\"ussman number, it means that $\Delta(t)$ grows more slowly and then $\delta_0^n$ is smaller. The following theorem verifies our intuition:
\begin{theorem}
\label{thm:Number-of-eigenvalues} Let $\omega$ be the Brjuno-R\"ussmann number satisfying Hypothesis H.1, or H.2, or H.3 and $L(E,\omega)>0$. Then, for any $x_{0}\in\TT$ and $E_{0}\in\mathbb{R}$,
\[
\#\left\{ E\in\mathbb{R}:\, f_{n}^{a}\left(x_{0},E,\omega\right)=0,\,\left|E-E_{0}\right|<\left(\delta_0^n\right)^{\frac{1}{h}}\right\} \le 13n\delta_0^n.
\]
\end{theorem}
~\\

We organize this paper as follows. In Section 2, we prove Theorem \ref{sbet}, the strong Birkhoff Ergodic theorem for the subharmonic function shifting on $\mathbb{T}$ with our Brjuno-R\"ussmann frequency. We apply it to the analytic quasi-periodic Jacobi operator and obtain Theorem \ref{ldtforfan}, Theorem \ref{ldth1} and Theorem \ref{ldth2} in Section 3, which are all the LDTs for $f_n^a(x,E,\omega)$ with different hypothesises. Then, we prove Theorem \ref{thm:Number-of-eigenvalues}, an application of them, in the last section.

\section{Strong Birkhoff Ergodic Theorem for Subharmonic Functions with the Brjuno-R\"ussmann shift}
Let $\{x\}=x-[x]$. For any positive integer $q$, complex number $\zeta=\xi+i\eta$ and $0\leq x<1$, define
 \begin{equation}
  F_{q,\zeta}(x)=\sum_{0\leq k<q}\log |\{x+k\omega\}-\zeta|\ \mbox{and}\  I(\zeta)=\int_0^1\log |y-\zeta |dy.
 \end{equation}
Let $|\{x+k_0\omega\}-\xi|=\min_{k=1}^{q_s} |\{x+k\omega\}-\xi|$, where $q_s$ is the denominator of the continued fraction approximants.
In \cite{GS01}, Goldstein and Schlag proved Lemma 3.1 that for any irrational $\omega$, there exists an absolute constant $C$ such that
 \begin{equation}\label{lemir}
   \left |F_{q_s,\zeta}(x)-q_sI(\zeta)\right |\leq C\log q_s+\left |\log |\{x+k_0\omega\}-\zeta|\right |.
 \end{equation}
Then
\begin{lemma}\label{lemlq}
 For any irrational $\omega$,
  \begin{equation}
  |F_{l_sq_s,\zeta}(x)-l_sq_s I(\zeta)|<Cl_s\log q_s+| \log D(x-\xi,-\omega,l_sq_s)|+ 2l_s\log q_{s+1},
\end{equation}
where $D(x,\omega,n):=\min_{k=0}^{n-1}\{x+k\omega\}$.
\end{lemma}
\begin{proof}
Define $x_h=x+hq_s \omega$ and $|\{x_h+k_h\omega\}-\xi|=\min_{k=0}^{q_s-1}|\{x_h+k\omega\}-\xi|$.
Due to (\ref{lemir}), we have
\begin{equation}\label{FnIn}
  |F_{l_sq_s,\zeta}(x)-l_sq_s I(\zeta)|\leq \sum_{h=0}^{l_s-1}  |F_{q_s,\zeta}(x_h)-l_sq_s I(\zeta)|     \leq\sum_{h=0}^{l_s-1}\left |\log |\{x_h+k_h\omega\}-\zeta|\right |+Cl_s\log q_s.
\end{equation}

We declare that  if there exists $0\leq j<q_s$ such that $|\{x+j\omega\}-\xi|\leq \frac{1}{2q_s}-\frac{1}{q_{s+1}}$, then $j=k_0$.  Actually, if $|\{x+j\omega\}-\xi|\leq \frac{1}{2q_s}-\frac{1}{q_{s+1}}$ and $j \not =k_0$, then $|\{x+k_0\omega\}-\xi|\leq |\{x+j\omega\}-\xi|\leq \frac{1}{2q_s}-\frac{1}{q_{s+1}}$, which implies
\[\left|\{x+k_0\omega\}-\{x+j\omega\}\right|\leq \frac{1}{q_s}-\frac{2}{q_{s+1}}.\]
Due to (\ref{irtor0}), it has
\begin{equation}\label{irtor}
 \frac{k}{q_s(q_{s+1}+q_s)}<|k\omega-\frac{k p_s}{q_s}|<\frac{k}{q_sq_{s+1}}\leq \frac{1}{q_{s+1}},\ \ 0<k< q_s.
\end{equation}
Then,
\[\left |\{x+j\frac{p_s}{q_s}\}-\{x+k_0\frac{p_s}{q_s}\}\right |<\frac{1}{q_s}.\] It is a contraction. Thus, there is at  most  one integer $0\leq k_0<q_s$ such that $|\{x+k_0\omega\}-\xi|< \frac{1}{2q_s}-\frac{1}{q_{s+1}}$ and
\begin{equation}\label{10004}
 |\{x+k\omega\}-\xi|>\frac{1}{2q_s}-\frac{1}{q_{s+1}}>\frac{1}{4q_{s}},\ \ k\not =k_0.
\end{equation}

Due to (\ref{irtor0}) again, it has
\begin{equation}\label{10012}  \frac{1}{2q_{s+1}}<\frac{1}{q_s+q_{s+1}}<|q_s\omega-p_s|<\frac{1}{q_{s+1}}.\end{equation}
Define $Q=[\frac{q_{s+1}}{q_s}]$ and
 let $j$ be the  number such that $ |\{x_j+k_j\omega\}-\xi|<\frac{1}{4q_{s+1}}$.
 Then by (\ref{10012}) and the above declaration, we have for any $ j-2Q+1\leq h<j$ and $j<h\leq j+2Q-1$,
 \[ |\{x_h+k_h\omega\}-\xi|>\frac{1}{4q_{s+1}}.\]
 Thus there are at most one point which is small than $\frac{1}{4q_{s+1}}$.  Combining it with (\ref{FnIn}), we have
\begin{eqnarray}
  |F_{l_sq_s,\zeta}(x)-l_sq_s I(\zeta)|&\leq &| \log D(x-\xi,-\omega,l_sq_s)|+Cl_s\log q_s +l_s\left |\log \frac{1}{4q_{s+1}}\right | \nonumber\\
  &\leq & | \log D(x-\xi,-\omega,l_sq_s)|+Cl_s\log q_s +2l_s\log q_{s+1} \nonumber.
\end{eqnarray}
\end{proof}

\begin{lemma}For any compact $\Omega\subset \IC$,
there exist constants  $\tilde c=\tilde c(\omega)$ and $\tilde C=\tilde C(\Omega,\omega)$ such that for any $n\ge 1$, $\zeta\in \Omega$ and  $0<\sigma\leq \tilde c$, we have
\begin{equation}
  \label{uni-int-equ}\int_0^1\exp(\sigma |F_{n,\zeta}(x)-n I(\zeta)|)dx \le  \exp\left
( \tilde  C\sigma n\breve \delta_0^n  \right ).
\end{equation}
where
\begin{equation}\label{set-br-delta0n}
  \breve\delta_0^n=\left\{\begin{array}
    {cccc}
    \breve \delta_{H.1}(n):=&\frac{\log\Delta(n)}{\left(\Delta^{-1}(C_{\omega}n)\right)^{1-}},&\ &\mbox{if}\ \omega\ \mbox{satisfies H.1},\\
     \breve \delta_{H.2}(n):=&\frac{1}{\left[\log(\Delta^{-1}(C_\omega n))\right]^{1-}},&\ &\mbox{if}\ \omega\ \mbox{satisfies H.2},\\
     \breve \delta_{H.3}(n):=&\frac{\log ( C_{\omega}n)}{\left[\Delta^{-1}(C_{\omega}n)\right]^{1-}} ,&\ &\mbox{if}\ \omega\ \mbox{satisfies H.3}.
  \end{array}\right.
\end{equation}
\end{lemma}
\begin{proof}
We first apply Lemma 3.2 in \cite{GS01}. It says that if $\Omega\subset \mathbb{T}$ is an arbitrary finite set, then for any $0<\sigma<1$,
\begin{equation}\label{intdis}
\int_{\mathbb{T}}\exp\left (\sigma |\log dist (x,\Omega)|\right )dx\leq \frac{2^{\sigma}}{1-\sigma}(\sharp \Omega)^{\sigma}.
\end{equation}
Set $\Omega=\{-m\omega:0\leq m<l_sq_s\}$. Then $\sharp \Omega=l_sq_s$ and $dist (x-\xi,\Omega)=D(x-\xi,-\omega,l_sq_s)$. Thus, by (\ref{intdis}),
  \[\int_{\mathbb{T}}\exp\left (\sigma |\log D(x-\xi,-\omega,l_sq_s)|\right )dx=\int_{\mathbb{T}}\exp\left (\sigma |\log dist (x,\Omega)|\right )dx\leq  \frac{2^{\sigma}}{1-\sigma} (l_sq_s)^{\sigma}.\]
  By Lemma \ref{lemlq}, we have
  \[ \int_{\mathbb{T}}\exp\left (\sigma|F_{l_sq_s,\zeta}(x)-l_sq_s I(\zeta)|\right )dx\leq \exp\left(2C\sigma \log (l_sq_s) + C\sigma l_s\log q_s+2\sigma l_s\log q_{s+1}\right)<\exp\left(5\sigma l_s\log q_{s+1}\right). \]
Now for any  $n$, there exist  $q_s$  and $q_{s+1}$ such that $q_s\leq n<q_{s+1}$. Let $n=l_sq_{s}+r_s$, where
$l_s=[\frac{n}{q_{s}}]$, $ 0\leq r_s=n-l_sq_s<q_{s}$. Then,
\begin{eqnarray*}
&&\int_0^1\exp(\sigma |F_{n,\zeta}(x)-n I(\zeta)|)dx \\
&\leq & \left
[\int_0^1\exp(2 \sigma |F_{l_sq_{s},\zeta}(x)-l_sq_{s}
I(\zeta)|)dx\right ]^{\frac{1}{2}}\times \left [\int_0^1\exp(2 \sigma
|F_{r_s,\zeta}(x)-r_s
I(\zeta)|)dx\right ]^{\frac{1}{2}}\nonumber\\
&\leq & \exp(5\sigma l_s\log q_{s+1}) \left [\int_0^1\exp(2\sigma |F_{r_s,\zeta}(x)-r_s
I(\zeta)|)dx\right ]^{\frac{1}{2}}.\nonumber
\end{eqnarray*} Let $r_{s-i+1}=l_{s-i}q_{s-i}+r_{s-i}$, where
$l_{s-i}=[\frac{r_{s-i+1}}{q_{s-i}}]$, $ 0\leq r_{s-i}=r_{s-i+1}-l_{s-i}q_{s-i}<q_{s-i}$. Then
\begin{eqnarray*}
  && \left
[\int_0^1\exp(2^{i}\sigma |F_{r_{s-i+1},\zeta}(x)-r_{s-i+1} I(\zeta)|)dx\right
]^{\frac{1}{2^i}}\\
&\leq & \left
[\int_0^1\exp(2^{i+1}\sigma |F_{l_{s-i}q_{s-i},\zeta}(x)-l_{s-i}q_{s-i}
I(\zeta)|)dx\right ]^{\frac{1}{2^{i+1}}}
\cdot \left [\int_0^1\exp(2^{i+1}\sigma
|F_{r_{s-i},\zeta}(x)-r_{s-i}
I(\zeta)|)dx\right ]^{\frac{1}{2^{i+1}}}\nonumber\\
&\leq & \exp\left (5\sigma l_{s-i}\log q_{s-i+1}  \right)\times \left
[\int_0^1\exp(2^{i+1}\sigma |F_{r_{s-i},\zeta}(x)-r_{s-i} I(\zeta)|)dx\right
]^{\frac{1}{2^{i+1}}}\nonumber.
\end{eqnarray*}
Note that for any irrational $\omega$, the denominators of its continued  fraction approximates satisfy
\[q_{n+1}=a_{n+1}q_n+q_{n-1}>2q_{n-1}.\]
Thus
\[q_n> 2^{m}q_{n-2m}.\]
 Therefore, if
$\zeta\in \Omega'$, where $\Omega'$ is a compact subregion of $\mathbb{C}$, then
\begin{eqnarray}
\int_0^1\exp(\sigma |F_{n,\zeta}(x)-n I(\zeta)|)dx &\leq & \exp\left
[ 5\sigma \left(l_s\log q_{s+1}+l_{s-1}\log q_{s}+\cdots+l_{s-2m+1}\log q_{s-2m+2} \right) \right ]\nonumber\\
&&\ \ \times \left [\int_0^1\exp(2^{2m+1} \sigma
|F_{r_{s-2m},\zeta}(x)-r_{s-2m}
I(\zeta)|)dx\right ]^{\frac{1}{2^{2m+1}}}\nonumber\\
&\leq & \exp\left
\{  5\sigma \left[\frac{n}{q_s}\log q_{s+1}+\sum_{j=1}^{2m-1}\frac{q_{s-j+1}}{q_{s-j}}\log q_{s-j+1}\right]+\hat C(\zeta)\sigma q_{s-2m}\right \}\nonumber\\
&\leq & \exp\left
\{  5\sigma \left[\frac{n}{q_s}\log q_{s+1}+\sum_{j=1}^{2m-1}\frac{q_{s-j+1}}{q_{s-j}}\log q_{s-j+1}+C'2^{-m}\sigma q_{s}\right]\right \}\label{int01},
\end{eqnarray}where $C'=C'(\Omega)$.~\\

Assume that $\omega$ satisfies Hypothesis H.1. Then, we assert that $\frac{t\log t}{\Delta^{-1}\left(C_{\omega}t\right)}$ is  monotone increasing. Indeed, let $y=\Delta^{-1}\left(C_{\omega}t\right)$. Then, due to the hypothesis that
\[y\Delta'(y)> \Delta(y),\]
it yields that
\[\Delta^{-1}\left(C_{\omega}t\right)\Delta'\left(\Delta^{-1}\left(C_{\omega}t\right)\right)> C_{\omega}t\ge C_{\omega}t-\frac{C_{\omega}t}{\log t+1}.\]
Combining it with the fact that
\[\frac{1}{\left(\Delta^{-1}\right)'\left(C_{\omega}t\right)}=\Delta'\left(\Delta^{-1}\left(C_{\omega}t\right)\right),\]
we have
\begin{equation}
  \label{numerator}
  \Delta^{-1}\left(C_{\omega}t\right)\left(\log t+1\right)>C_{\omega}t \Delta'\left(\Delta^{-1}\left(C_{\omega}t\right)\right)\log t.
\end{equation}
Now we finish the proof of the assertion as (\ref{numerator}) shows that the numerator of the derivative of  $\frac{t\log t}{\Delta^{-1}\left(C_{\omega}t\right)}$ is positive.

Due to (\ref{brn}) and (\ref{irtor0}),
\begin{equation}\label{qq}
  C_{\omega}q_i<\Delta(q_{i-1}) \ \mbox{and}\ q_{i-1}>\Delta^{-1}(C_{\omega}q_i).
\end{equation}
Therefore,
\[\frac{q_i}{q_{i-1}}\log q_i<\frac{q_i}{\Delta^{-1}(C_{\omega}q_i)}\log q_i.\]
We apply the assertion and obtain
\begin{equation}
  \label{sum-qi}\sum_{j=1}^{2m-1}\frac{q_{s-j+1}}{q_{s-j}}\log q_{s-j+1}<(2m-1)\frac{q_s}{\Delta^{-1}(C_{\omega}q_s)}\log q_s.
\end{equation}
Recall that $\Delta(t)$ is monotone increasing and continuous. And so is  $\Delta^{-1}(t)$. Combining it with (\ref{qq}), we have
\begin{equation}
  \label{qsdelta-1}
  q_s>\Delta^{-1}(C_\omega q_{s+1})>\Delta^{-1}(C_\omega n),
\end{equation}
and for any $n>n_0(\omega)$,
\begin{equation}
  \label{logqs1}l_{s}\log q_{s+1}\le \frac{n}{q_s}\log\left(\frac{\Delta(q_s)}{C_\omega}\right)\le \frac{2n\log \Delta(q_s)}{q_s}\le \frac{2n\log \Delta(q_s)}{\Delta^{-1}(C_\omega n)}\le \frac{2n\log \Delta(n)}{\Delta^{-1}(C_\omega n)}.
\end{equation}
Choose $m\sim \log_{  2}\Delta^{-1}(C_{\omega}n)$. Then
\[C' 2^{-m}q_s\lesssim \frac{n}{\Delta^{-1}(C_{\omega}n)}.\]
Combining it with (\ref{int01}), (\ref{sum-qi}), (\ref{logqs1}) and the assertion that $\frac{t\log t}{\Delta^{-1}\left(C_{\omega}t\right)}$ is  monotone increasing, we have
\begin{eqnarray}
  \int_0^1\exp(\sigma |F_{n,\zeta}(x)-n I(\zeta)|)dx &\leq &  \exp\left
\{  5\sigma \left[\frac{2n\log \Delta(n)}{\Delta^{-1}(C_\omega n)}+(2m-1)\frac{q_s}{\Delta^{-1}(C_{\omega}q_s)}\log q_s+\frac{n}{\Delta^{-1}(C_{\omega}n)}\right]\right \}\nonumber\\
&\leq & \exp\left
\{  5\sigma \left[\frac{2n\log \Delta(n)}{\Delta^{-1}(C_\omega n)}+2\log_{2}\Delta^{-1}(C_{\omega}n)\frac{n}{\Delta^{-1}(C_{\omega}n)}\log n+\frac{n}{\Delta^{-1}(C_{\omega}n)}\right]\right \}\nonumber\\
&\leq & \exp\left
\{ \tilde  C\sigma n\log\Delta(n)\frac{\log_{\bar C}\Delta^{-1}(C_{\omega}n)}{\Delta^{-1}(C_{\omega}n)}\right \}\nonumber\\
&\le &\exp\left
\{ \tilde  C\sigma n\frac{\log\Delta(n)}{\left(\Delta^{-1}(C_{\omega}n)\right)^{1-}}\right \}.\label{int-h1}
\end{eqnarray}

Assume that $\omega$ satisfies  Hypothesis H.2 which implies that $\Delta(t)<\exp\left(\frac{t}{\log t}\right)$ holds. Then
\begin{equation}\label{qsqs}
  \frac{\log \Delta(q_s)}{q_s}<\frac{1}{\log q_s}\ \mbox{and}\ \frac{\log n}{\Delta^{-1}(C_\omega n)} <  \frac{1}{\log \Delta^{-1}(C_\omega n)}.
\end{equation}
Combining them with (\ref{sum-qi}) and (\ref{qsdelta-1}),
\[l_{s}\log q_{s+1}<\frac{2n\log \Delta(q_s)}{q_s}<\frac{2n}{\log q_s}<\frac{2n}{\log(\Delta^{-1}(C_\omega n))},\]
\[\sum_{j=1}^{2m-1}\frac{q_{s-j+1}}{q_{s-j}}\log q_{s-j+1}<(2m-1)\frac{q_s}{\Delta^{-1}(C_{\omega}q_s)}\log q_s<(2m-1)\frac{n}{\Delta^{-1}(C_{\omega}n)}\log n< (2m-1)\frac{n}{\log(\Delta^{-1}(C_\omega n))}.\]
Let $m\sim \log_{ 2}\log \Delta^{-1}(C_\omega n)$. Thus,
\begin{equation}
  \label{barcmqs}2^{-m}q_s\le \frac{n}{\log \Delta^{-1}(C_\omega n)},
\end{equation}
and
\begin{eqnarray}
  \int_0^1\exp(\sigma |F_{n,\zeta}(x)-n I(\zeta)|)dx &\le&  \exp\left
\{ \tilde  C\sigma  \log_{\bar C}\log \Delta^{-1}(C_\omega n)\frac{n}{\log \Delta^{-1}(C_\omega n)}  \right \}\nonumber \\&\le& \exp\left
\{ \tilde  C\sigma \frac{n}{\left[\log(\Delta^{-1}(C_\omega n))\right]^{1-}}\right\}.\label{int-h2}
\end{eqnarray}

Assume the $\omega$ satisfies Hypothesis H.3 which implies  that $\frac{\log \Delta(t)}{t}$ is non-increasing. Due to (\ref{qsdelta-1}) and (\ref{logqs1}),
\begin{equation}
  \label{h3-logqs1}l_{s}\log q_{s+1}\le  \frac{2n\log \Delta(q_s)}{q_s}\le \frac{n\log \Delta(\Delta^{-1}(C_\omega n))}{\Delta^{-1}(C_\omega n)}=\frac{n\log ( C_{\omega}n)}{\Delta^{-1}(C_{\omega}n)}.
\end{equation}
By (\ref{sum-qi}),
\begin{equation}
  \label{h3-qi}\sum_{j=1}^{2m-1}\frac{q_{s-j+1}}{q_{s-j}}\log q_{s-j+1}<(2m-1)\frac{q_s}{\Delta^{-1}(C_{\omega}q_s)}\log q_s\le (2m-1)\frac{n}{\Delta^{-1}(C_{\omega}n)}\log n.
\end{equation}
Choose $m\sim \log_{2}(\Delta^{-1}(C_{\omega}n))$. Thus,
\[
 2^{-m}q_{s}\sim \frac{q_s}{\Delta^{-1}(C_{\omega}n)}<\frac{n\log n}{\Delta^{-1}(C_{\omega}n)}.
\]
Combining them, we have
\begin{equation}
  \label{h3}\int_0^1\exp(\sigma |F_{n,\zeta}(x)-n I(\zeta)|)dx \le  \exp\left
( \tilde  C\sigma  \frac{n\log ( C_{\omega}n)}{\left[\Delta^{-1}(C_{\omega}n)\right]^{1-}}  \right ).
\end{equation}
\end{proof}
\begin{remark}\label{remforhar}
  Note that $\log |x|$ is a subharmonic function. Thus, if $h$ is a 1-periodic harmonic function defined on a neighborhood of real axis, then for any positive $n$ and  $0<\sigma\leq \tilde c$, we have
\begin{equation}\label{12014}
 \int_0^1\exp(\sigma | \sum_{k=1}^n h(\{x+k\omega\})-n \int_0^1h(y)dy|)dx<\exp\left(\tilde C\sigma n\breve\delta_0^n\right).
\end{equation}
\end{remark}~\\

To finish the proof of Theorem \ref{sbet}, we need  the following Riesz's theorem proved in \cite{GS08}:
\begin{lemma}
\label{lem:riesz} Let $u:\Omega\to \IR$ be a subharmonic function on
a domain $\Omega\subset\IC$. Suppose that $\partial \Omega$ consists
of finitely many piece-wise $C^1$ curves. There exists a positive
measure $\mu$ on~$\Omega$ such that for any $\Omega_1\Subset \Omega$
(i.e., $\Omega_1$ is a compactly contained subregion of~$\Omega$),
\begin{equation}
\label{eq:rieszrep} u(z) = \int_{\Omega_1}
\log|z-\zeta|\,d\mu(\zeta) + h(z),
\end{equation}
where $h$ is harmonic on~$\Omega_1$ and $\mu$ is unique with this
property. Moreover, $\mu$ and $h$ satisfy the bounds \be
\mu(\Omega_1) &\le& C(\Omega,\Omega_1)\,(\sup_{\Omega} u - \sup_{\Omega_1} u), \label{21002} \\
\|h-\sup_{\Omega_1}u\|_{L^\infty(\Omega_2)} &\le&
C(\Omega,\Omega_1,\Omega_2)\,(\sup_{\Omega} u - \sup_{\Omega_1} u)
\label{21003} \ee for any $\Omega_2\Subset\Omega_1$.
\end{lemma}
\begin{proof}[The proof of Theorem \ref{sbet}]
Notice that the ergodic measure for the shift on the Torus is the Lebesgue measure and $m(\mathbb{T})=1$. Then, $<u>=\int_{\mathbb{T}}u(x)dx$, and
\[\sum_{k=1}^n u(x+k\omega)-n<u>=\sum_{k=1}^n\int_{\Omega'} \log |\{x+k\omega\}-\zeta|d\mu(\zeta)-n\int_{\Omega'} I(\zeta)d\mu(\zeta)+\sum_{k=1}^nh(\{x+k\omega\})-n\int_0^1h(y)dy.\]
Recall that
\[\sum_{k=1}^n\int_{\Omega'} \log |\{x+k\omega\}-\zeta|d\mu(\zeta)=\int_{\Omega'} F_{n,\zeta}(x)d\mu(\zeta).\]
Then
\begin{eqnarray}
  \int_0^1\exp\left (\sigma |\sum_{k=1}^n u(x+k\omega)-n<u>|\right )dx & \leq &\left [\int_0^1\exp\left (\sigma \left |\int_{\Omega'} (F_{n,\zeta}(x)-n I(\zeta))d\mu(\zeta)\right |\right )dx\right ]^{\frac{1}{2}}\nonumber\\
   &&\ \ \times \left [\int_0^1\exp\left (\sigma \left |\sum_{k=1}^nh(\{x+k\omega\})-n\int_0^1h(y)dy\right |\right )dx\right ]^{\frac{1}{2}}\nonumber.
\end{eqnarray}
Since $\exp(\sigma \cdot)$ is a convex function, the Jensen's inequality implies that
\begin{eqnarray}
  \int_0^1\exp\left (\sigma \left |\int_{\Omega'} (F_{n,\zeta}(x)-n I(\zeta))d\mu(\zeta)\right |\right )dx &\leq & \int_0^1\int_{\Omega'}\exp\left (\sigma \mu(\Omega')\left | F_{n,\zeta}(x)-n I(\zeta)\right |\right )\frac{d\mu(\zeta)}{\mu(\Omega')}dx\nonumber\\
  &=&\int_{\Omega'}\int_0^1\exp\left (\sigma \mu(\Omega')\left | F_{n,\zeta}(x)-n I(\zeta)\right |\right )dx\frac{d\mu(\zeta)}{\mu(\Omega')}\nonumber\\
  &\leq & \int \exp\left(\tilde C\sigma \mu(\Omega')n\breve\delta_0^n\right)\frac{d\mu(\zeta)}{\mu(\Omega')}\nonumber\\
  &\leq &\exp\left(\tilde C\sigma \mu(\Omega')n\breve\delta_0^n\right).\nonumber
 \end{eqnarray}
Thus, combining it with (\ref{12014}), we have for any $0<\sigma\leq \frac{\tilde c}{\mu(\Omega')}$,
\[\int_0^1\exp\left (\sigma |\sum_{k=1}^n u(x+k\omega)-n\langle u\rangle|\right )dx< \exp\left(\tilde C\sigma \mu(\Omega')n\breve\delta_0^n\right).\]
Recall the Markov's inequality: For any  measurable extended real-valued function $f(x)$ and $\epsilon >0$,we have
\[\mes \left (\{x\in \mathbb{X}:|f(x)|\ge \epsilon\} \right ) \leq \frac{1}{\epsilon}\int_{\mathbb{X}} |f|dx.\]
Let $f(x)=\exp\left (\sigma |\sum_{k=1}^n u(x+k\omega)-n\langle u\rangle|\right )$ and $\epsilon=\exp(\sigma \breve\delta n)$, then
\begin{eqnarray}
 &&\mes\left (\left \{x\in \mathbb{X}:|\sum_{k=1}^n u(x+k\omega)-n\langle u\rangle|>\delta n\right \}\right )\nonumber\\
 &=& \mes \left (\left \{x\in \mathbb{X}:\exp\left (\sigma |\sum_{k=1}^n u(x+k\omega)-n\langle u\rangle|\right ) \ge \exp(\sigma \breve\delta n)\right \} \right )\nonumber\\
 &\leq & \exp\left(-\sigma \delta n+\tilde C\sigma \mu(\Omega')n\breve\delta_0^n\right).\nonumber
\end{eqnarray}
We finish this proof by setting $\breve C=2\tilde C \mu(\Omega')$.
\end{proof}

\section{Large Deviation Theorems for $f_n^a(x,E,\omega)$}

To  apply  Theorem \ref{sbet}, we first need to define some subharmonic functions. Let
\begin{equation}\label{Mnax}
  M_n^a(x,E,\omega):=\left(\prod_{j=1}^na(x+j\omega)\right)M_n(x,E,\omega)=\prod_{j=1}^n\left ( \begin{array}{cc}
  v(x+j\omega)-E & \overline{a(x+j\omega)} \\
 a(x+(j+1)\omega)& 0 \\
  \end{array}\right ).
\end{equation}
Note that a real function $f(x)$ on $\mathbb{T}$ has its complex analytic extension $f(z)$ on the complex strip $\mathbb{T}_{\rho}=\{z:|\Im z|<\rho\}$ and the complex analytic extension of $\bar a(x)$ should be defined on  $\mathbb{T}_{\rho}$ by
\[\tilde a(z):=\overline{a(\frac{1}{z})}.\]
Then, the extension of $M^a_n(x,E,\omega)$ is
\begin{equation}\label{Mnaz}
  M_n^a(z,E,\omega)=\prod_{j=1}^n\left ( \begin{array}{cc}
  v(z+j\omega)-E & \tilde a(z+j\omega) \\
 a(z+(j+1)\omega)& 0 \\
  \end{array}\right ),
\end{equation}where $z+\omega$ means $z\exp\left(2\pi i\omega\right)$ here. Moreover, simple computations yield that
\begin{equation}\label{Mnazdet}
 M_n^a(z,E,\omega)=\left ( \begin{array}{cc}
  f_n^a(z,E,\omega) & \tilde a(z)f^a_{n-1}(z+\omega,E,\omega) \\
 a(z+n\omega)f^a_{n-1}(z,E,\omega)& -\tilde a(z)a(z+n\omega)f^a_{n-2}(z+\omega,E,\omega) \\
  \end{array}\right ),
\end{equation}
where
\[\begin{aligned}
    f^a_n(z, E,\omega)  &= \det\bigl(H_n(z, \omega)-E\bigr)\\
    & =
    \begin{vmatrix}
     v\bigl(z+\omega\bigr) - E & -a\bigl(z+2\omega\bigr) & 0  &\cdots &\cdots & 0\\[5pt]
    -\tilde a\bigl(z+2\omega\bigr) & v\bigl(z+2\omega\bigr) - E & -a\bigl(z+3\omega\bigr) & 0 & \cdots & 0\\[5pt]
    \vdots & \vdots & \vdots & \vdots & \vdots & \vdots\\
    &&&&&-a\bigl(z+n\omega\bigr) \\[5pt]
    0 & \dotfill & 0 &0 &-\tilde a\bigl(z+n\omega\bigr) & v\bigl(z+n\omega\bigr) - E
    \end{vmatrix}
    \end{aligned}.\]
Note that if $\Im z=0$, then $H_n(z, \omega)=H_n(x,\omega)$ is Hermitian. Now with fixed $E$ and $\omega$, the function $\frac{1}{N}\log\|M_N^a(z,E,\omega)\|$ is subharmonic. In this paper, we only need to consider $E\in \mathscr{E}$, where
\[
\mathscr{E}:=[-2\|a(x)\|_{L^{\infty}(\mathbb{T})}-\|v(x)\|_{L^{\infty}(\mathbb{T})},\ 2\|a(x)\|_{L^{\infty}(\mathbb{T})}+\|v(x)\|_{L^{\infty}(\mathbb{T})}],\]
as the spectrum $\cS_{\omega}\subset \mathscr{E}$. Thus, for any irrational $\omega$ and $1\leq n\in\mathbb{N}$,
\begin{equation}\label{supmna}
  \sup\limits_{E\in\mathscr{E},x\in\mathbb{T}}\frac{1}{N}\log\|M_N^a(z,E,\omega)\|\leq M_0,
\end{equation}
where
\[
M_0:=\log\left(3\|a\|_{L^{\infty}(\mathbb{T}_{\rho})}+2 \|v\|_{L^{\infty}(\mathbb{T}_{\rho})}\right).
\]

We also need to define the unimodular matrix
\begin{equation}
  \label{mnu}M^u_n(x,E,\omega):=\frac{M_n(x,E,\omega)}{|\det
M_n(x,E,\omega)|^{\frac{1}{2}}},
\end{equation} which makes sense a.e. $x\in \mathbb{T}$ and has the
 relationship
\[\int_{\mathbb{T}}\frac{1}{n}\log\|M_n^u(x,E,\omega)\|dx=L_n(E,\omega).\]
Then, we have the LDTs for the matrices as follows:
\begin{lemma}\label{ldtformatrices}Let $\omega$ be the Brjuno-R\"ussmann number satisfying Hypothesis H.1, or H.2, or H.3 and $L(E,\omega)>0$.
There exist $\hat c=\hat c(v,a,\omega)$ and $\check c=\check c(v,a,\omega)$ such that for any $n\ge 0$ and $\delta>\delta_0^n$,
\begin{equation}
  \mes\left\{x:\left|u(x,E,\omega)-\langle u\rangle \right|> \delta\right\}<\exp\left(- \hat  c\delta n\right)+\exp(-\check c\delta^2n),\label{sldteforuna}
\end{equation}
where $u(x,E,\omega)$ can be $\frac{1}{n}\log\|M_n^a(x,E,\omega)\|$, $\frac{1}{n}\log\|M_n(x,E,\omega)\|$ and $\frac{1}{n}\log\|M_n^u(x,E,\omega)\|$. What's more, there exists $\bar c=\bar c(a,v,\omega)$ such that if $\delta=\kappa L(E,\omega)$ with $\kappa<\frac{1}{10}$, then the exception measure in (\ref{sldteforuna}) will be less than $\exp\left(-\bar c\kappa^2L(\omega,E)n\right)$.
\end{lemma}
\begin{proof}When $u=\frac{1}{n}\log\|M_n^a(x,E,\omega)\|$,
The  LDT (\ref{sldteforuna}) is about the analytic matrix; when $u=\frac{1}{n}\log\|M_n(x,E,\omega)\|$, The  LDT (\ref{sldteforuna}) is about the Jacobi cocycles; when $u=\frac{1}{n}\log\|M^u_n(x,E,\omega)\|$, The  LDT (\ref{sldteforuna})  is about the unimodular matrix to satisfy the hypothesises of Lemma \ref{expanding_contracting_directions}, Lemma \ref{wedge_triangle_ineq} and the Avalanche Principle(Proposition \ref{prop:AP}). In \cite{T18}, our second author obtained these  LDTs with finite Liouville frequency, which means that $\beta(\omega)<\infty$. Due to the  fact that  $\beta(\omega)=0$ for any Brjuno-R\"ussmann number, the proofs in that paper are also available here.
\end{proof}
 What's more,  the following lemma shows that $L_n(E,\omega)$ and $L_n^a(E,\omega)=\langle \frac{1}{n}\log\|M_n^a(x,E,\omega)\| \rangle $ in the above LDTs can be exchanged by $L(E,\omega)$ and $L^a(E,\omega)$, respectively. Here,
\begin{equation}
  \label{la}
  L^a(E,\omega)=\lim_{n\to\infty}L_n^a(E,\omega)=L(E,\omega)+D,
\end{equation}
and
\begin{equation}
  \label{setting-d}
  D:=\int_{\mathbb{T}}\log|a(x)|dx=\int_{\mathbb{T}}\log|\bar a(x)|dx.
\end{equation}
\begin{lemma}\label{Ln_L} Let $L(E,\omega)>0$. For any integer $n>1$, we have
\[
0\le L_{n}-L=L_{n}^{u}-L^{u}=L_{n}^{a}-L^{a}<C_{0}\frac{\left(\log n\right)^{2}}{n}
\]
where $C_{0}=C_{0}\left(a,v,\omega,E\right)$.\end{lemma}
\begin{proof}
 It is the same as  Lemma 3.9 in \cite{BV13}, which was for the strong Diophantine $\omega$. They applied the same LDTs, whose $\delta_0^n=\frac{(\log n)^A}{n}$ with that frequency, to obtain its proof. It is  available here, since it only need the fact, which our LDTs also satisfy,  that $\delta_0^n$ is much less than the positive Lyapunov exponent.
\end{proof}

Although the details of the proof of Lemma \ref{ldtformatrices} can be found in \cite{T18}, we still give a brief introduction here, to make the readers understand the methods we apply in this section to obtain Theorem \ref{ldtforfan}-\ref{ldth2}. Easy computations show that the differences between $\frac{1}{N}\log\|M_N^a(x,E,\omega)\|$ and $\frac{1}{N}\log\|M_N(x,E,\omega)\|$ and between $\frac{1}{N}\log\|M_N^a(x,E,\omega)\|$ and $\frac{1}{N}\log\|M_n^u(x,E,\omega)\|$ are constructed by the combination of $\frac{1}{N}\sum_{j=1}^N\log|a(x+j\omega)|$ and $\frac{1}{N}\sum_{j=1}^N\log|\bar a(x+j\omega)|$, whose complex extensions can be estimated by our Theorem \ref{sbet} easily. Therefore, we only need to prove the LDT for $\frac{1}{N}\log\|M_n^a(x,E,\omega)\|$, which also has a subharmonic extension. Due to this subharmonicity,
\begin{equation}
  \mes\left\{x:\left|\frac{1}{n}\sum_{j=1}^n\log\|M_n^a(x+j\omega,E,\omega)\|-L^a_n(E,\omega)\right|> \delta\right\}<\exp\left(-   c\delta n\right).\label{sbetforuna}
\end{equation}
On the other hand, for any $k\in\mathbb{Z}$,
\begin{eqnarray*}
  -\frac{2M_0k}{n}+\sum_{j=0}^{k-1}\frac{k-j}{nk}d(x+j\omega)&\leq &\frac{1}{n}\log\|M_n^a(x,E,\omega)\|-\frac{1}{kn}\sum_{j=1}^k\log\|M_n^a(x+j\omega,E,\omega)\|\\
&\leq &\frac{2M_0k}{n}-\sum_{j=0}^{k-1}\frac{k-j}{nk}d(x+(n+j-1)\omega),
\end{eqnarray*}
where $d(x)=\log|a(x+\omega)\bar a(x)|$.
Obviously, it also can be solved by our Theorem \ref{sbet}. Now, we can explain why we apply the BMO norm and the John-Nirenberg inequality, not the method for (\ref{sldteforuna}), to obtain the LDTs for $f_n^a(x,E,\omega)$. The reason is that Theorem \ref{sbet} holds for $f_n^a(x,E,\omega)$, but we can not handle the difference between $\frac{1}{n}\log|f_n^a(x,E,\omega)|$ and $\frac{1}{kn}\sum_{j=1}^k\log|f_n^a(x+j\omega,E,\omega)|$.

We will apply  the analyticity of $f_n^a(x,E,\omega)$ and the  subharmonicity of $\frac{1}{n}\log|f_n^a(x,E,\omega)|$ via the following lemmas in this paper.
\begin{defi}
  Let $H>1$. For any arbitrary subset $\cB\subset \cD(z_0,1)\subset \mathbb{C}$ we say $\cB\in Car_1(H,K)$ if $\cB\subset \bigcup_{j=1}^{j_0}\cD(z_j,r_j)$ with $j_0<K$, and
  \begin{equation}
    \label{cardef}
    \sum_{j}r_j<e^{-H}.
  \end{equation}Here $D(z,r)$ means the complex platform center at $z$ with radius $r$. If $d$ is a positive integer greater than one and $\cB\subset \prod_{i=1}^d\subset \IC^d$ then we define inductively that $\cB\in Car_d(H,K)$ for any $z\in \IC\backslash \cB_j$, here $\cB_z^{(j)}=\{(z_1,\cdots,z_d)\in \cB:z_j=z\}.$
\end{defi}
\begin{lemma}[Cartan estimate, Lemma 2.4 in \cite{GS11}]\label{carest}
Let $\phi(z_1,\cdots,z_d)$ be an analytic function defined in a polydisk $\cP=\prod_{j=1}^d\cD(z_{j,0},1),\ z_{j,0}\in \IC$. Let $M\ge \sup_{\underline z\in \cP}\log|\phi(\underline z)|,\ m\le \log|\phi(\underline z_0)|,\ \underline z_0=(z_{1,0},\cdots,z_{d,0})$. Given $H\gg 1$ there exists a set $\cB\subset \cP, \cB\in Car_d(H^{\frac{1}{d}},K),\ K=C_dH(M-m)$, such that
\[
  \log|\phi(\underline z)|>M-C_dH(M-m)
\]
for any $\underline z\in\prod_{j=1}^d\cD(z_{j,0},\frac{1}{6})\backslash \cB$.
\end{lemma}
\begin{lemma}[Lemma 2.4 in  \cite{GS08}]
\label{sh_better_upper_bound} Let
$u$ be a subharmonic function defined on $\cA_{\rho}$ such that
$\sup_{\cA_{\rho}}u\le M$. There exist constants $C_{1}=C_{1}\left(\rho\right)$
and $C_{2}$ such that, if for some $0<\delta<1$ and some $L$ we
have
\[
\mes\left\{ x\in\TT:\, u\left(x\right)<-L\right\} >\delta,
\]
then
\[
\sup_{\TT}u\le C_{1}M-\frac{L}{C_{1}\log\left(C_{2}/\delta\right)}.
\]
\end{lemma}~\\

Recalling the definitions of $M_n(x,E,\omega)$, $M_n^a(x,E,\omega)$, $M_n^u(x,E,\omega)$ and the expression (\ref{Mnazdet}), we have
\begin{equation}
\label{Mnzdet}
M_n(z,E,\omega)=\left\{\begin{array}
  {cc}f_n(z,E,\omega)&-\frac{\tilde a(z)}{a(z+\omega)}f_{n-1}(z+\omega,E,\omega)\\
  f_{n-1}(z,E,\omega)& -\frac{\tilde a(z)}{a(z+\omega)}f_{n-2}(z+\omega,E,\omega)
\end{array}\right\},
\end{equation}
and
\begin{equation}
  \label{Mnuzdet}
 M_n^u(z,E,\omega)=\left\{\begin{array}
  {cc}f_n^u(z,E,\omega)&-\frac{\tilde a(z)}{a(z+\omega)}\left|\frac{a(z+\omega)}{\tilde a(z)}\right|^{\frac{1}{2}}f^u_{n-1}(z+\omega,E,\omega)\\
  \left|\frac{a(z+n\omega)}{\tilde a(z+(n-1)\omega)}\right|^{\frac{1}{2}}f^u_{n-1}(z,E,\omega)& -\frac{\tilde a(z)}{a(z+\omega)}\left|\frac{a(z+n\omega)a(z+\omega)}{\tilde a(z)\tilde a(z+(n-1)\omega)}\right|^{\frac{1}{2}}f^u_{n-2}(z+\omega,E,\omega)
\end{array}\right\},
\end{equation}
where
\begin{equation}
  \label{MnzdetMnaz}
  f_n(z,E,\omega)=\frac{1}{\prod_{j=1}^na(z+j\omega)}f_n^a(z,E,\omega),
\end{equation}
and
\begin{equation}
  \label{MnuzdetMnaz}
  f^u_n(z,E,\omega)=\frac{1}{\left|\prod_{j=0}^{n-1}a(z+(j+1)\omega)\tilde a(z+j\omega)\right|^{\frac{1}{2}}}f_n^a(z,E,\omega)=\left(\prod_{j=0}^{n-1}\left|\frac{a(z+(j+1)\omega)}{\tilde a(z+j\omega)}\right|^{\frac{1}{2}}\right)f_n(z,E,\omega).
\end{equation}

Assume $L(E,\omega)=\gamma>0$. Then, we can obtain a particular deviation theorem as follow:
\begin{lemma}
\label{lem:considerable-difficulty-lemma}There exists $l_{0}=l_{0}\left(a,v,\gamma\right)$
such that
\[
\mes\left\{ x\in\TT:\,\left|f_{l}\left(x\right)\right|\le\exp\left(-l^{3}\right)\right\} \le\exp\left(-l\right)
\]
for all $l\ge l_{0}$.\end{lemma}
\begin{proof}
 It is the same as  Lemma 4.2 in \cite{BV13}, which was for the strong Diophantine $\omega$. We have the same reason which we just stated in the  proof of Lemma \ref{Ln_L} to omit  this proof.
\end{proof}
Note that in order to simplify the notation, we suppressed the dependence on $E$ and $\omega$. We will be doing this throughout this paper if there is no confusion.
According to Lemma \ref{sh_better_upper_bound} and \ref{lem:considerable-difficulty-lemma}, we can have more choices of the deviation and the exceptional measure.
\begin{lemma}
\label{small_fl_is_small}Let $\sigma>0$ and $g(n)>0$. There exist constants
$l_{0}=l_{0}\left( a, v,\gamma \right)$ and $n_{0}=n_{0}\left( a, v,\gamma\right)$
such that
\[
\mes\left\{ x\in\TT:\,\left|f_{l}\left(x\right)\right|\le\exp\left(-g(n)\right)\right\} \le\exp\left(-g(n)l^{-3}\right)
\]
for any $n\ge n_{0}$ and for any $l_{0}\le l \lesssim g(n)$. The
same result, but with possibly different $l_{0}$ and $n_{0}$, holds
for $f^u_{l}$.\end{lemma}
\begin{proof}
Assume
\[
\mes\left\{ x\in\mathbb{T}:\,\left|f_{l}\left(x\right)\right|\le\exp\left(-g(n)\right)\right\} >\exp\left(-g(n)l^{-3}\right).
\]
 We have that
\begin{align*}
\left|f^a_{l}\left(x\right)\right| & =\left|f_{l}\left(x\right)\right|\prod_{j=1}^{l}\left|a\left(x+j\omega\right)\right|\le\exp\left(-g(n)\right)C^{l-1}
\le\exp\left(-\frac{1}{2}g(n)\right)
\end{align*}
on a set of measure greater than $\exp\left(-g(n)l^{-3}\right)$.
By  Lemma \ref{sh_better_upper_bound}, it implies that for any $x\in\mathbb{T}$,
\[
\left|f^a_{l}\left(x\right)\right|\le\exp\left(C_{1}l-\frac{g(n)}{2C_{1}\log\left(C_{2}\exp\left(g(n)l^{-3}\right)\right)}\right)\le\exp\left(-Cl^{3}\right)
.\]
Due to Theorem \ref{sbet},
\[\mes\left (\left \{x\in \mathbb{T}:|\sum_{k=1}^l \log|a(x+k\omega)-lD|>l\right \}\right ) \leq  \exp\left(-cl\right).\]
Therefore, recalling (\ref{MnzdetMnaz}), we have
\[
\left|f_{l}\left(x\right)\right|\le\exp\left(l(1-D)-C'l^{3}\right)\le\exp\left(-Cl^{3}\right)
\]
for all $x$ except for a set of measure less than $\exp\left(-cl\right)$.
It contradicts with the previous lemma. At last, by (\ref{MnuzdetMnaz}), we can prove the result for $f^u_l$ by similar methods.
\end{proof}

Now we need some facts about stability of contracting
and expanding directions of unimodular matrices. It follows from the
polar decomposition that if $A\in SL\left(2,\mathbb{C}\right)$ then there
exist unit vectors $u_{A}^{+}\perp u_{A}^{-}$ and $v_{A}^{+}\perp v_{A}^{-}$
such that $Au_{A}^{+}=\left\Vert A\right\Vert v_{A}^{+}$ and $Au_{A}^{-}=\left\Vert A\right\Vert ^{-1}v_{A}^{-}$.
\begin{lemma}[Lemma 2.5 in \cite{GS08}]
\label{expanding_contracting_directions}
For any $A$, $B\in SL\left(2,\mathbb{C}\right)$ we have
\begin{align*}
\left|Bu_{AB}^{-}\wedge u_{A}^{-}\right| & \le\left\Vert A\right\Vert ^{-2}\left\Vert B\right\Vert ,\,\left|u_{BA}^{-}\wedge u_{A}^{-}\right|\le\left\Vert A\right\Vert ^{-2}\left\Vert B\right\Vert ^{2}\\
\left|v_{AB}^{+}\wedge v_{A}^{+}\right| & \le\left\Vert A\right\Vert ^{-2}\left\Vert B\right\Vert ^{2},\,\left|v_{BA}^{+}\wedge Bv_{A}^{+}\right|\le\left\Vert A\right\Vert ^{-2}\left\Vert B\right\Vert .
\end{align*}
\end{lemma}
\begin{lemma}[Lemma 4.5 in \cite{BV13}]
\label{wedge_triangle_ineq} If $A\in SL\left(2,\mathbb{C}\right)$ and
$w_{1}$, $w_{2}$, and $w_{3}$ are unit vectors in the plane then
\[
\left|w_{1}\wedge Aw_{2}\right|\le\left|w_{1}\wedge Aw_{3}\right|+\sqrt{2}\left\Vert A^{-1}\right\Vert \left|w_{2}\wedge w_{3}\right|
\]
and
\[
\left|w_{1}\wedge Aw_{2}\right|\le\left|w_{3}\wedge Aw_{2}\right|+\sqrt{2}\left\Vert A\right\Vert \left|w_{1}\wedge w_{3}\right|
\]
\end{lemma}

Now, we can improve the Lemma \ref{small_fl_is_small}, but the LDT is about three determinants.
\begin{lemma}
\label{three_determinants}  There exist constants $0<\kappa=\kappa(\omega)<1$, $0<\tau=\tau(\omega)<1$,
$l_{0}=l_{0}\left( a, v,\gamma \right)$ and $n_{0}=n_{0}\left( a, v,\gamma\right)$ such that
\begin{equation}\label{three_determinants_estimate}
\mes\left\{ x\in\TT:\,\left|f^u_{n}\left(x\right)\right|+\left|f^u_{n}\left(x+j_{1}\omega\right)\right|+\left|f^u_{n}\left(x+j_{2}\omega\right)\right|\le\exp\left(nL_{n}-100n\delta_0^n\right)\right\}
\le\exp\left(-n^{1-\kappa}\right)
\end{equation}
for any $l_{0}\le j_{1}\le j_{1}+l_{0}\le j_{2}\le n^{\tau}$
and $n\ge n_{0}$.\end{lemma}
\begin{proof}
Here we assume $\delta_0^n\ge n^{-\frac{1}{3}}$, since the proof of Lemma 4.6 in \cite{BV13} can be applied without any change when $\delta\le n^{-\frac{1}{3}}$.

For any $1\le j\le n$, due to  Lemma \ref{ldtformatrices} and \ref{Ln_L}, choose the deviation $\delta=\frac{n}{j}\delta_0^n>\delta_0^j$ and then
\begin{equation}\label{sldteforulu}
 \mes\left\{x:\left|\log\|M_l^u(x)\|-jL\right|> n\delta_0^n\right\}<\exp\left(- \hat  c\delta j\right)+\exp\left(- \check  c\delta^2 j\right)\le 2\exp\left(-  \check  c\left(\delta_0^n\right)^2 n\right).
  \end{equation}
Let $\mathcal{G}_{n}$ be the set of points $x\in\TT$ such that for any $1\le j\le n$
and $\left|l\right|\le 2n$,
$$\left|\log\left\Vert M^u_{j}\left(x+l\omega\right)\right\Vert -jL\right|\le n\delta_0^n,$$
and $$\left|\log\left|a\left(x+j\omega\right)\right|-D\right|\le n\delta_0^n.$$  Due to (\ref{sldteforulu}) and Theorem \ref{sbet} for $\log|a(x)|$, we have that
\[\mes\left(\TT\setminus\cG_{n}\right)\le 4n^2\exp\left(-  \check  c\left(\delta_0^n\right)^2 n\right)\le \exp\left(-c\left(\delta_0^n\right)^2 n\right).\]
 Note that $\det M_l^u(x,E,\omega)\equiv 1$. Therefore, for any $x,\ E$ and $\omega$,
$$\left\|M_l^u(x,E,\omega)\right\|=\left\|\left(M_l^u\right)^{-1}(x,E,\omega)\right\|.$$

Let $\left\{ e_{1},e_{2}\right\} $ be the standard basis of $\mathbb{R}^{2}$ and for any integer $j$,  $u_j^+$, $u_j^1$, $v_j^+$ and $v_j^-$ be the unit vectors satisfying $u_{j}^{+}\perp u_{j}^{-}$, $v_{j}^{+}\perp v_{j}^{-}$, $M_j^uu_{j}^{+}=\left\Vert M_j^u\right\Vert v_{j}^{+}$ and $M_j^uu_{j}^{-}=\left\Vert M_j^u\right\Vert ^{-1}v_{j}^{-}$. Then
\begin{eqnarray*}
f^u_{n}\left(x\right)&=&M^u_{n}\left(x\right)e_{1}\wedge e_{2}
=\left(M^u_{n}\left(x\right)\left[\left(u_{n}^{+}\left(x\right)\cdot e_{1}\right)u_{n}^{+}\left(x\right)+\left(u_{n}^{-}\left(x\right)\cdot e_{1}\right)u_{n}^{-}\left(x\right)\right]\right)\wedge e_{2}\\
&=&\left(u_{n}^{+}\left(x\right)\cdot e_{1}\right)\left\Vert M^u_{n}\left(x\right)\right\Vert v_{n}^{+}\left(x\right)\wedge e_{2}+\left(u_{n}^{-}\left(x\right)\cdot e_{1}\right)\left\Vert M^u_{n}\left(x\right)\right\Vert ^{-1}v_{n}^{-}\left(x\right)\wedge e_{2}.
\end{eqnarray*}
If $\left|f^u_{n}\left(x\right)\right|\le\exp\left(nL_n-100n\delta_0^n\right)$,
then
\[
\left\Vert M^u_{n}\left(x\right)\right\Vert \left|u_{n}^{+}\left(x\right)\cdot e_{1}\right|\left|v_{n}^{+}\left(x\right)\wedge e_{2}\right|-\left\Vert M^u_{n}\left(x\right)\right\Vert ^{-1}\left|u_{n}^{-}\left(x\right)\cdot e_{1}\right|\left|v_{n}^{-}\left(x\right)\wedge e_{2}\right|
\le\exp\left(nL_n-100n\delta_0^n\right).
\]
Due to Lemma \ref{Ln_L}, for any $x\in \mathcal{G}_{n}$,
\begin{eqnarray*}
  \left|u_{n}^{-}\left(x\right)\wedge e_{1}\right|\left|v_{n}^{+}\left(x\right)\wedge e_{2}\right|&\le & \left\Vert M^u_{n}\left(x\right)\right\Vert^{-1}\exp\left(nL_n-100n\delta_0^n\right)+\left\Vert M^u_{n}\left(x\right)\right\Vert ^{-2}\\
  &\le &\exp\left(n\left(L_{n}-L\right)-99n\delta_0^n\right)+\exp\left(2n\delta_0^n-2nL\right)\\
  &\le &\exp\left(-90n\delta_0^n\right).
\end{eqnarray*}
Hence,
\begin{equation}
  \left|u_{n}^{-}\left(x\right)\wedge e_{1}\right|\le\exp\left(-40n\delta_0^n\right)\ \mbox{or}\ \left|v_{n}^{+}\left(x\right)\wedge e_{2}\right|\le\exp\left(-40n\delta_0^n\right).
\end{equation}

Suppose (\ref{three_determinants_estimate}) fails. Let $\sigma<\kappa<1/2$. Recall $n\delta_0^n\ge n^{1-2\sigma}$ and set
\[\tilde{\cG}_n:=\left\{ x\in\cG_{n}:\,\left|f^u_{n}\left(x\right)\right|+\left|f^u_{n}\left(x+j_{1}\omega\right)\right|+
\left|f^u_{n}\left(x+j_{2}\omega\right)\right|\le\exp\left(nL_{n}-100n\delta_0^n\right)\right\}.\]
We have
\[
\mes \tilde{\cG}_n
>\exp\left(-n^{1-\kappa}\right)-\exp\left(-n\delta_0^n\right)>\frac{1}{2}\exp\left(-n^{1-\kappa}\right).\]
If $x\in \tilde{\cG}_n$, then either
$\left|u_{n}^{-}\left(x\right)\wedge e_{1}\right|\le\exp\left(-40n\delta_0^n\right)$
or $\left|v_{n}^{+}\left(x\right)\wedge e_{2}\right|\le\exp\left(-40n\delta_0^n\right)$
has to hold for two of the points $x$, $x+j_{1}\omega$, $x+j_{2}\omega$.

We first assume that
\begin{equation}
\left|u_{n}^{-}\left(x+j_{1}\omega\right)\wedge e_{1}\right|\le\exp\left(-40n\delta_0^n\right)\quad\text{and}\quad\left|u_{n}^{-}\left(x+j_{2}\omega\right)\wedge e_{1}\right|\le\exp\left(-40n\delta_0^n\right).\label{first_catch}
\end{equation}
From Lemma \ref{wedge_triangle_ineq} and Lemma \ref{expanding_contracting_directions}, we have that if $x\in \cG_n$, then
\begin{eqnarray*}
&&\left|u_{n}^{-}\left(x+j_{2}\omega\right)\wedge M^u_{j_{2}-j_{1}}\left(x+j_{1}\omega\right)u_{n}^{-}\left(x+j_{1}\omega\right)\right|\\
&\le&\left|u_{n}^{-}\left(x+j_{2}\omega\right)\wedge M^u_{j_{2}-j_{1}}\left(x+j_{1}\omega\right)u_{n+j_{2}-j_{1}}^{-}\left(x+j_{1}\omega\right)\right|\\
&&
\ \ \ \ \ \ \ \ \ \ \ \ \ \ \ \ \ \ +C\left\Vert\left (M^u_{j_{2}-j_{1}}\right)^{-1}\left(x+j_{1}\omega\right)\right\Vert \left|u_{n+j_{2}-j_{1}}\left(x+j_{1}\omega\right)\wedge u_{n}^{-}\left(x+j_{1}\omega\right)\right|\\
&=&\left|u_{n}^{-}\left(x+j_{2}\omega\right)\wedge M^u_{j_{2}-j_{1}}\left(x+j_{1}\omega\right)u_{n+j_{2}-j_{1}}^{-}\left(x+j_{1}\omega\right)\right|\\
&&
\ \ \ \ \ \ \ \ \ \ \ \ \ \ \ \ \ \ +C\left\Vert M^u_{j_{2}-j_{1}}\left(x+j_{1}\omega\right)\right\Vert \left|u_{n+j_{2}-j_{1}}^{-}\left(x+j_{1}\omega\right)\wedge u_{n}^{-}\left(x+j_{1}\omega\right)\right|\\
&\le&\left\Vert M^u_{n}\left(x+j_{2}\omega\right)\right\Vert ^{-2}\left\Vert M^u_{j_{2}-j_{1}}\left(x+j_{1}\omega\right)\right\Vert\\
&&
\ \ \ \ \ \ \ \ \ \ \ \ \ \ \ \ \ \ +C\left\Vert M^u_{j_{2}-j_{1}}\left(x+j_{1}\omega\right)\right\Vert \left\Vert M^u_{n}\left(x+j_{1}\omega\right)\right\Vert ^{-2}\left\Vert M^u_{j_{2}-j_{1}}\left(x+\left(n+j_{1}\right)\omega\right)\right\Vert ^{2}\nonumber\\
&\le&\exp\left(\left(-2n+j_{2}-j_{1}\right)L+3n\delta_0^n\right)+C\exp\left(\left(-2n+3\left(j_{2}-j_{1}\right)\right)L+5n\delta_0^n\right)\nonumber\\
&\le &\exp\left(-nL\right)\nonumber.
\end{eqnarray*} Combined it with  Lemma \ref{wedge_triangle_ineq} and (\ref{first_catch}), we obtain
\begin{eqnarray*}
&&\left|e_{1}\wedge M^u_{j_{2}-j_{1}}\left(x+j_{1}\omega\right)e_{1}\right|\\
&\le&\left|e_{1}\wedge M^u_{j_{2}-j_{1}}\left(x+j_{1}\omega\right)u_{n}^{-}\left(x+j_{1}\omega\right)\right|
+C\left\Vert M^u_{j_{2}-j_{1}}\left(x+j_{1}\omega\right)^{-1}\right\Vert \left|e_{1}\wedge u_{n}^{-}\left(x+j_{1}\omega\right)\right|\\
&\le&\left|u_{n}^{-}\left(x+j_{2}\omega\right)\wedge M^u_{j_{2}-j_{1}}\left(x+j_{1}\omega\right)u_{n}^{-}\left(x+j_{1}\omega\right)\right|
+C\left\Vert M^u_{j_{2}-j_{1}}\left(x+j_{1}\omega\right)\right\Vert \left|e_{1}\wedge u_{n}^{-}\left(x+j_{2}\omega\right)\right|\\
&&\ \ \ \ \ \ \ \ \ \ \ \ \ \ \ \ \ \ +C\left\Vert M^u_{j_{2}-j_{1}}\left(x+j_{1}\omega\right)^{-1}\right\Vert \left|e_{1}\wedge u_{n}^{-}\left(x+j_{1}\omega\right)\right|\\
&\le&\exp\left(-nL\right)+C\exp\left(\left(j_{2}-j_{1}\right)L-39n\delta_0^n\right)+C\exp\left(-39n\delta_0^n\right)\\
&\le&\exp\left(-30n\delta_0^n\right).
\end{eqnarray*}
Due to the fact that
\[
\left|e_{1}\wedge M^u_{j_{2}-j_{1}}\left(x+j_{1}\omega\right)e_{1}\right|=\left|\frac{a\left(x+j_{2}\omega\right)}{a\left(x+\left(j_{2}-1\right)\omega\right)}\right|^{1/2}\left|f^u_{j_{2}-j_{1}-1}\left(x+j_{1}\omega\right)\right|,
\]
and the setting of $\cG_n$, we have
\[
\left|f^u_{j_{2}-j_{1}-1}\left(x+j_{1}\omega\right)\right|\le C\exp\left(\frac{1}{2}\left(n\delta_0^n-D\right)-30n\delta_0^n\right)\le\exp\left(-20n\delta_0^n\right).
\]
Similarly, we can obtain
\[
\left|f^u_{j_{2}-j_{1}-1}\left(x+\left(n+j_{1}+1\right)\omega\right)\right|\le\exp\left(-20n\delta_0^n\right),
\]
if we  assume that
\begin{equation}\label{first_catch1}
  \left|v_{n}^{+}\left(x+j_{1}\omega\right)\wedge e_{2}\right|\le\exp\left(-40n\delta_0^n\right)\quad\text{and}\quad\left|v_{n}^{+}\left(x+j_{2}\omega\right)\wedge e_{2}\right|\le\exp\left(-40n\delta_0^n\right).
\end{equation}
What's more, the same type of estimates are obtained if we replace $\left(j_{1},j_{2}\right)$ in (\ref{first_catch}) and (\ref{first_catch1}) with $\left(0,j_{1}\right)$ or $\left(0,j_{2}\right)$.

In conclusion
\[
\mes\left\{ x\in\mathbb{T}:\left|f^u_{l}\left(x\right)\right|\le\exp\left(-20n\delta_0^n\right)\right\} >\frac{1}{2}\exp\left(-n^{1-\kappa}\right)
\]
for some choice of $l$ from $j_{1}-1$, $j_{2}-1$, $j_{2}-j_{1}-1$. However, choosing  $g(n)=20n\delta_0^n$ in Lemma \ref{small_fl_is_small} and $\tau=\frac{\kappa-\sigma}{4}$ in the hypothesis of this lemma, we have
\[
\mes\left\{ x\in\mathbb{T}:\left|f^u_{l}\left(x\right)\right|\le\exp\left(-20n\delta_0^n\right)\right\}
\le\exp\left(-20n\delta_0^nl^{-3}\right)\le\exp\left(-20n^{1-\sigma-3\tau}\right)\ll\exp\left(-cn^{1-\kappa}\right).
\]
Thus, we complete the proof by this contradiction.
\end{proof}

One of our methods to obtain a large deviation estimate for a single determinant is the $BMO(\mathbb{T})$ norm. $BMO(\mathbb{T})$ is the space of functions of bounded mean oscillation on $\mathbb{T}$. Identifying functions that differ only by an additive constant, then norm on $BMO(\mathbb{T})$ is given by
\begin{equation}
  \label{bmonorm}
  \|f\|_{BMO(\mathbb{T})}:=\sup_{I\subset \mathbb{T}}\frac{1}{|I|}\int_I|f-\langle f\rangle_I|dx,
\end{equation}
where $\langle f\rangle_I:=\int_If(x)dx$. Applying the previous lemma, we obtain the following lower bound of the mean value of $\frac{1}{N}\left|f^u_{N}\left(x\right)\right|$, which will help us estimate the  BMO norm.
\begin{lemma}
\label{avg_entries->lower_bound} There exist constants $0<c_0=c_0(\omega)\le 1$
and $n_{0}=n_{0}\left(a,v,\gamma\right)$
such that for $n\ge n_{0}$ we have
\[
\int_{\TT}\frac{1}{n}\left|f^u_{n}\left(x\right)\right|dx>L_{n}-\left(\delta_0^n\right)^{c_0}.
\]
\end{lemma}
\begin{proof}
Set
\begin{eqnarray*}
  &&\Omega_n:=\bigg\{x\in \mathcal{G}_n:\min
  \Big\{\left|f^u_{n}\left(x+j_1\omega\right)\right|+\left|f^u_{n}\left(x+j_{2}\omega\right)\right|+\left|f^u_{n}\left(x+j_{3}\omega\right)\right|:\\
  &&\ \ \ \ \ \ \ \ \ \ \ \ \ \ \ \ \ \ \ \ \ \ \ \ \ \ \ 0<j_1<j_1+l_0\le j_2<j_2+l_0\le j_3\le n^\kappa\Big\}>\exp\left(nL_n-100n\delta_0^n\right)\bigg\}.
\end{eqnarray*}
Then,
$\mes\left(\mathbb{T}\setminus\Omega_{n}\right)\le n^\tau\exp\left(-n^{1-\kappa}\right)<\exp\left(-\frac{1}{2}n^{1-\kappa}\right)$.

Define $\nu_n^u\left(x\right)=\log\left|f^u_{n}\left(x\right)\right|/n$ and
set $M=\left[\frac{n^{\tau}}{l_{0}}\right]\ge n^{\frac{\tau}{2}}$ for large $n$. For any $x\in\Omega_{n}$
we have that $\nu_n^u\left(x+kl_{0}\omega\right)>L_{n}-100\delta_0^n-\frac{\log 3}{n}$
for all but at most two $k$'s, $1\le k\le M$. We have
\begin{multline}
\left\langle \nu_n^u\right\rangle :=\int_{\mathbb{T}}\nu_n^u\left(x\right)dx=\frac{1}{M}\sum_{k=1}^{M}\int_{\mathbb{T}}u\left(x+kl_{0}\omega\right)dx\\
\ge\int_{\Omega_{n}}\left(\frac{M-2}{M}\left(L_{n}-100\delta_0^n-\frac{\log3}{n}\right)+\frac{2}{M}\inf_{1\le k\le M}\nu_n^u\left(x+kl_{0}\omega\right)\right)dx
+\frac{1}{M}\sum_{k=1}^{M}\int_{\mathbb{T}\setminus\Omega_{n}}\nu_n^u\left(x+kl_{0}\omega\right)dx.\label{tu_first_estimate}
\end{multline}

Define $\nu_n^a\left(x\right)=\log\left|f^a_{n}\left(x\right)\right|/n$. Note that $\nu_n^a(x)$ can be extended to the complex trip $\mathbb{T}_{\rho}$ where $\nu_n^a(z)$ is subharmonic. Due to (\ref{Mnazdet}), we
have that
\[
S:=\sup_{z\in\mathcal{A}_{\rho_{0}}}\nu_n^a\left(z\right)\le\sup_{z\in\mathcal{A}_{\rho_{0}}}\frac{1}{n}\log\left\Vert M^a_{n}\left(z\right)\right\Vert <M_0.
\]
Applying Cartan's estimate, Lemma \ref{carest}, to $f^a_n(z)$ with $M=Sn$, $m=<\nu_n^a>n$ and $H=n^{\frac{\tau}{4}}$, we have
\begin{equation}
\inf_{1\le k\le M}\nu_n^a\left(x+kl_{0}\omega\right)\ge S-C\left(S-\left\langle v\right\rangle \right)n^{\frac{\tau}{4}}>-C\left(2|S|-\left\langle \nu_n^a\right\rangle \right)n^{\frac{\tau}{4}}\label{inf_u}
\end{equation}
up to a set not exceeding $CM\exp\left(-n^{\frac{\tau}{4}}\right)$ in measure.
Combining it with the relationship that
\begin{equation}
\nu_n^u\left(x\right)=\nu_u^a\left(x\right)-\frac{1}{2n}\left(\sum_{j=1}^n+\sum_{j=0}^{n-1}\right)\log|a(x+j\omega)|\label{u...v}
\end{equation}and applying (\ref{inf_u}) and Theorem \ref{sbet} for $\displaystyle \frac{1}{2n}\left(\sum_{j=1}^n+\sum_{j=0}^{n-1}\right)\log|a(x+j\omega)|$ with deviation $|D|$, we have
\[
\inf_{1\le k\le M}\nu_n^u\left(x+kl_{0}\omega\right)>-C\left(2|S|-\left\langle \nu_n^a\right\rangle \right)n^{\frac{\tau}{4}}-2|D|>-C'n^{\frac{\tau}{4}}
\]
up to a set $\mathcal{B}_{n}$ not exceeding $CM\exp\left(-n^{\frac{\tau}{4}}\right)+\exp(-\hat c|D|n)<\exp\left(-\frac{1}{2}n^{\frac{\tau}{4}}\right)$
in measure. Therefore,
\[
\left\langle \nu_n^u\right\rangle \ge\left(1-\frac{2}{M}\right)\left(L_{n}-100\delta_0^n-\frac{\log3}{n}\right)-\frac{C'n^{\frac{\tau}{4}}}{M}
-\frac{2}{M}\sum_{k=1}^{M}\int_{\Omega_{n}^{c}\cup\mathcal{B}_{n}}\left|\nu_n^u\left(x+kl_{0}\omega\right)\right|.
\]
Let $g(n)=n^3$ in Lemma \ref{small_fl_is_small}. Then simple calculations shows that  $\left\Vert \nu_n^u\right\Vert _{L^{2}\left(\mathbb{T}\right)}\le Cn^{3}$. Thus,
\[
\int_{\Omega_{n}^{c}\cup\mathcal{B}_{n}}\left|\nu_n^u\left(x+kl_{0}\omega\right)\right|dx\le\left(\mes\left\{ \Omega_{n}^{c}\cup\mathcal{B}_{n}\right\} \right)^{1/2}\left\Vert u\right\Vert _{L^{2}\left(\mathbb{T}\right)}\le Cn^{3}\exp\left(-\frac{1}{4}n^{\frac{\tau}{4}}\right)\le C\exp\left(-\frac{1}{8}n^{\frac{\tau}{4}}\right).
\]
Above all,
\begin{equation}\label{setting_c_0}
  \left\langle \nu_n^u\right\rangle \ge L_n-100\delta_0^n-\frac{2}{M}L_n-C'n^{-\frac{\tau}{4}}-C\exp\left(-\frac{1}{8}n^{\frac{\tau}{4}}\right)\ge L_n-\left(\delta_0^n\right)^{c_0}.
\end{equation}
\end{proof}
\begin{remark}\label{define_c_0}
 Due to the setting of $\tau$ and (\ref{setting_c_0}), easy computations shows that $$c_0=\left\{\begin{array}
   {cc}1,&\ \ \mbox{if}\ \Delta(t)>t^5;\\
   \frac{A}{5},&\ \ \mbox{if}\ \Delta(t)\sim t^A,\ 1<A<5.\\
 \end{array}\right. .$$
\end{remark}

We will show that the supermum of the subharmonic function $u_n^a(z,E,\omega)$ on $\mathbb{T}$ is closed to its mean value. Here, we will apply the property that a subharmonic function at a point is small than the its integration on the platform center at that point. From the proof of Theorem \ref{sbet}, it is easily seen that the sharp LDT for $u_n^a(x)$ can been extended to the complex region $\mathbb{T}_{\rho}$:
\begin{equation}
  \mes\{x:|u^a_n(re(x),E,\omega)-L^a_n(r,E,\omega)|> \delta\}<\exp\left(- \hat  c\delta n\right)+\exp\left(- \check  c\delta^2 n\right),\ \forall \delta>\delta_0^n,\label{sldtcomplex}
\end{equation}
where
\[L^a_n(r,E,\omega)=\int_{\mathbb{T}}u^a_n(re(x),E,\omega)dx.\]
Lemma 4.1 in \cite{GS08} proved that there exists $C_{0}=C_{0}\left(M_0,\rho\right)$ such that for any $r_{1},r_{2}\in\left(1-\rho,1+\rho\right)$ we have
\begin{equation}\label{L(r1)-L(r2)}
  |L^a_n(r_1)-L^a_n(r_2)|\leq C_0|r_1-r_2|.
\end{equation}
\begin{lemma}\label{M^a-upper-bound}
  For any integer $n>1$ we have that
\[
\sup_{x\in\TT}\log\|M_n^a(x)\|\le nL_{n}^{a}+2n\delta_0^n.
\]
\end{lemma}
\begin{proof}
Due to (\ref{sldtcomplex}) with $\delta=\delta_0^n$, we have
\[
\log\left\Vert M_n^a \left(re(x)\right)\right\Vert -nL_{n}^{a}\left(r\right)\le n\delta_0^n
\]
except for a set of measure less than $\exp\left(-\hat cn\delta_0^n\right)+\exp\left(- \check  c\left(\delta_0^n\right)^2 n\right)$. By the
subharmonicity of $\log\left\Vert M_n^a \left(z\right)\right\Vert $
we have
\begin{eqnarray}\label{eq:M^a-nL^a...upper_bound}
\log\left\Vert M_n^a \left(x\right)\right\Vert -nL_{n}^{a}&\le&\frac{1}{\pi n^{-2}}\int_{D\left(x,n^{-1}\right)}\left(\log\left\Vert M_n^a \left(z\right)\right\Vert -nL_{n}^{a}\right)dA\left(z\right)\\
&\le&\frac{1}{\pi n^{-2}}\int_{1-n^{-1}}^{1+n^{-1}}\int_{x-2n^{-1}}^{x+2n^{-1}}\left|\log\left\Vert M_n^a \left(ry\right)\right\Vert -L_{n}^{a}\right|rdydr.\nonumber
\end{eqnarray}
For $r\in\left(1-n^{-1},1+n^{-1}\right)$ we have
\begin{eqnarray*}
\int_{x-2n^{-1}}^{x+2n^{-1}}\left|\log\left\Vert M_n^a \left(ry\right)\right\Vert -L_{n}^{a}\right|dy
&\le&\int_{x-2n^{-1}}^{x+2n^{-1}}\left|\log\left\Vert M_n^a \left(ry\right)\right\Vert -L_{n}^{a}\left(r\right)\right|dy+\left|L_{n}^{a}-L_{n}^{a}\left(r\right)\right|\\
&\le& n\delta_0^n+C_{a}n\left[\exp\left(-\frac{\hat c}{2}n\delta_0^n\right)+\exp\left(- \frac{\check  c}{2}\left(\delta_0^n\right)^2 n\right)\right]+C_{3}n^{-1}<2n\delta_0^n.
\end{eqnarray*}
\end{proof}

Then, we will use the following lemma proved by Bourgain, Goldstein and Schlag in \cite{BGS01}, not the definition, to calculate the BMO norm of subharmonic functions.
\begin{lemma}[Lemma 2.3 in \cite{BGS01}]\label{bgsbmonorm}
  Suppose u is subharmonic on $\mathbb{T}_{\rho}$, with $\mu(\mathbb{T}_{\rho})+\sup_{z\in\mathbb{T}_{\rho}}h(z)\le n$ where $\mu(\mathbb{T}_{\rho})$ and $h(z)$ comes from Lemma \ref{lem:riesz}. Furthermore, assume that $u=u_0+u_1$, where
\begin{equation}
\|u_0-\langle u_0\rangle\|_{L^{\infty}(\mathbb{T})}\leq \epsilon_0\ \ \mbox{and}\ \ \|u_1\|_{L^1(\mathbb{T})}\leq \epsilon_1.
\end{equation}
Then for some constant $C_{\rho}$ depending only on $\rho$,
\[\|u\|_{BMO(\mathbb{T})}\leq C_{\rho}\left (\epsilon_0\log \left(\frac{n}{\epsilon_1}\right )+\sqrt{n\epsilon_1}\right).\]
\end{lemma}
\begin{lemma}
\label{ldt_entries_weak}There exist constant $c_1=c_1(a,v,E,\rho,\gamma)$ and absolute constant $C$
such that for every integer $n$ and any $\delta>0$ we have
\[
\mes\left\{ x\in\mathbb{T}:\,\left|\log\left|f^a_{n}\left(x\right)\right|-\left\langle \log\left|f^a_{n}\right|\right\rangle \right|>n\delta\right\} \le C\exp\left(-c_{1}\delta (\delta_0^n)^{-c_0}\right).
\]where $c_0$ comes from Remark \ref{define_c_0}.
The same estimate with possibly different $c_{1}$  holds
for $f^u_{n}$.\end{lemma}
\begin{proof}
It is enough to establish the estimate for $n$ large enough.
By  Lemma \ref{avg_entries->lower_bound}
and Lemma \ref{M^a-upper-bound},
\[
\begin{cases}
\left\langle \nu_n^a\right\rangle \ge L_{n}^{a}-\left(\delta_0^n\right)^{c_0}\\
\sup_{\mathbb{T}}\nu_n^a\le L_{n}^{a}+2\delta_0^n.
\end{cases}
\]
This implies that
\[
\left\Vert \nu_n^a-\left\langle \nu_n^a\right\rangle \right\Vert _{L^{1}\left(\TT\right)}\le 3\left(\delta_0^n\right)^{c_0}.
\]
Due to Lemma \ref{bgsbmonorm} with setting $\epsilon_0=0$, we have
\[
\left\Vert \nu_n^a\right\Vert _{BMO\left(\mathbb{T}\right)}=\left\Vert \nu_n^a-\left\langle \nu_n^a\right\rangle \right\Vert _{BMO\left(\mathbb{T}\right)}\le C_{\rho}\left\Vert \nu_n^a-\left\langle \nu_n^a\right\rangle \right\Vert _{L^{1}\left(\mathbb{T}\right)}^{1/2}\le 3C_{\rho}\left(\delta_0^n\right)^{c_0}.
\]
Then, the well-known John-Nirenberg inequality  tells us how to apply this MBO norm to obtain the large deviation theorem:
Let $f$ be a function of bounded mean oscillation on $\mathbb{T}$. Then there exist the absolute constants $C$ and $c$ such that for any $\gamma>0$
  \begin{equation}\label{jn}
    meas\{x\in\mathbb{T}: | f(x)-<f>|>\gamma \}\leq C\exp \left (-\frac{c\gamma}{\| f\|_{BMO}}\right ).
  \end{equation}
Thus,
\[
\mes\left\{ x\in\mathbb{T}:\,\left|\nu_n^a\left(x\right)-\left\langle \nu_n^a\right\rangle \right|>\delta\right\} \le C\exp\left(-c_1\delta (\delta_0^n)^{-c_0}\right).
\]
\end{proof}

Now, due to the above proof and Remark \ref{setting_c_0}, to prove Theorem \ref{ldtforfan}-\ref{ldth2}, the only thing we need to do is obtain $\left\|\frac{1}{n}\log|f_n^a|\right\|_{BMO}=O\left(\delta_0^n\right)$, when $\Delta(t)\sim t^A$ and $1<A<5$. In the following proof, we will use the Avalanche Principle to refine the previous estimation:
\begin{prop}[Avalanche Principle]
\label{prop:AP} Let $A_1,\ldots,A_n$ be a sequence of  $2\times
2$--matrices whose determinants satisfy
\begin{equation}
\label{eq:detsmall} \max\limits_{1\le j\le n}|\det A_j|\le 1.
\end{equation}
Suppose that \be
&&\min_{1\le j\le n}\|A_j\|\ge H>n\mbox{\ \ \ and}\label{large}\\
   &&\max_{1\le j<n}[\log\|A_{j+1}\|+\log\|A_j\|-\log\|A_{j+1}A_{j}\|]<\frac12\log H\label{diff}.
\ee Then
\begin{equation}
\Bigl|\log\|A_n\cdot\ldots\cdot A_1\|+\sum_{j=2}^{n-1}
\log\|A_j\|-\sum_{j=1}^{n-1}\log\|A_{j+1}A_{j}\|\Bigr| <
C\frac{n}{H} \label{eq:AP}
\end{equation}
with some absolute constant $C$.
\end{prop}

\begin{proof}[The Proof of Theorem \ref{ldtforfan} to  \ref{ldth2}]
Define
\[
\left[\begin{array}{cc}
f^u_{n}\left(x\right) & 0\\
0 & 0
\end{array}\right]=\left[\begin{array}{cc}
1 & 0\\
0 & 0
\end{array}\right]M^u_{n}\left(x\right)\left[\begin{array}{cc}
1 & 0\\
0 & 0
\end{array}\right]=:\cM_{n}^{u}\left(x\right).
\]
and $\cM_{n}^{a}$ analogously. Obviously, $\left|f_{n}^{a}\left(x\right)\right|=\left\Vert \cM_{n}^{a}\left(x\right)\right\Vert $. Let $c'$ be a small constant constant, $l\sim n^{c'}$ be an integer and  $n=l+\left(m-2\right)l+l'$
with $2l\le l'\le3l$. Set $A_{j}^{u}\left(x\right)=M^u_{l}\left(x+\left(j-1\right)l\omega\right)$,
$j=2,\ldots,m-1$,
\[
A_{1}^{u}\left(x\right)=M^u_{l}\left(x\right)\left[\begin{array}{cc}
1 & 0\\
0 & 0
\end{array}\right]=\left[\begin{array}{cc}
f^u_{l}\left(x\right) & 0\\
\star & 0
\end{array}\right],
\]
and
\[
A_{m}^{u}\left(x\right)=\left[\begin{array}{cc}
1 & 0\\
0 & 0
\end{array}\right]M^u_{l'}\left(x+(m-1)l\omega\right)=\left[\begin{array}{cc}
f^u_{l'}\left(x+(m-1)l\omega\right) & \star\\
0 & 0
\end{array}\right].
\] The matrices $A^a_j$ have similar definitions. By Lemma \ref{ldtformatrices}, for any $j=2,\ldots,m-1$,
\[ \mes\left\{x:\left|\frac{1}{l}\log\|A_j(x)\|-L_l\right|> \frac{1}{20}L_l\right\}<\exp\left(- cL_ll\right).\]
And due to the fact that
\[
\log\left|f^u_{l}\left(x\right)\right|\le\log\left\Vert A_{1}^{u}\left(x\right)\right\Vert \le\log\left\Vert M^u_{l}\left(x\right)\right\Vert ,
\]
Lemma \ref{ldt_entries_weak},\ref{avg_entries->lower_bound} and \ref{ldtformatrices}, we have
\[ \mes\left\{x:\left|\frac{1}{l}\log\|A_1(x)\|-L_l\right|> \frac{1}{10}L_l\right\}<\exp\left(- cL_l\left(\delta_0^l\right)^{-c_0}\right),\]
and an analogous estimate for $\log\left\Vert A_{m}^{u}\right\Vert $. Now the hypothesis of Avalanche Principle
are satisfied and hence
\begin{equation}\label{apforau}
\log\left\Vert \cM_{n}^{u}\left(x\right)\right\Vert +\sum_{j=2}^{m-1}\log\left\Vert A_{j}^{u}\left(x\right)\right\Vert -\sum_{j=1}^{m-1}\log\left\Vert A_{j+1}^{u}\left(x\right)A_{j}^{u}\left(x\right)\right\Vert =O\left(\frac{1}{l}\right)
\end{equation}
up to a set of measure less than $3m\exp\left(- cL_l\left(\delta_0^l\right)^{-c_0}\right)$. By the definitions of $M_n^u$ and $M_n^a$, easy computations show that
\begin{eqnarray*}
  &&\log\left\Vert \cM_{n}^{u}\left(x\right)\right\Vert +\sum_{j=2}^{m-1}\log\left\Vert A_{j}^{u}\left(x\right)\right\Vert -\sum_{j=1}^{m-1}\log\left\Vert A_{j+1}^{u}\left(x\right)A_{j}^{u}\left(x\right)\right\Vert\\
  &=&\log\left\Vert \cM_{n}^{a}\left(x\right)\right\Vert +\sum_{j=2}^{m-1}\log\left\Vert A_{j}^{a}\left(x\right)\right\Vert -\sum_{j=1}^{m-1}\log\left\Vert A_{j+1}^{a}\left(x\right)A_{j}^{a}\left(x\right)\right\Vert.
\end{eqnarray*}
Thus, (\ref{apforau}) also holds for $\cM_{n}^{a}$. If we set
\[
u_{0}\left(x\right)=\log\left\Vert A_{m}^{a}\left(x\right)A_{m-1}^{a}\left(x\right)\right\Vert +\log\left\Vert A_{2}^{a}\left(x\right)A_{1}^{a}\left(x\right)\right\Vert,
\]
then  the previous relation can be rewritten as
\[
\log\left\Vert \cM_{n}^{a}\left(x\right)\right\Vert +\sum_{j=2}^{m-1}\log\left\Vert M^a_{l}\left(x+\left(j-1\right)l\omega\right)\right\Vert \\
-\sum_{j=2}^{m-2}\log\left\Vert M^a_{2l}\left(x+\left(j-1\right)l\omega\right)\right\Vert -u_{0}\left(x\right)=O\left(\frac{1}{l}\right).
\]
Similarly,for any $0\le k<l-1$,
\[
\log\left\Vert \cM_{n}^{a}\left(x\right)\right\Vert +\sum_{j=2}^{m-1}\log\left\Vert M^a_{l}\left(x+k\omega+\left(j-1\right)l\omega\right)\right\Vert \\
-\sum_{j=2}^{m-2}\log\left\Vert M^a_{2l}\left(x+k\omega+\left(j-1\right)l\omega\right)\right\Vert -u_{k}\left(x\right)=O\left(\frac{1}{l}\right),
\]
where
\[u_k(x)=\log\left\| \left[\begin{array}{cc}
1 & 0\\
0 & 0
\end{array}\right]M^a_{l'-k}\left(x+k\omega+(m-1)l\omega\right)\cdot    A_{m-1}^a(x+k\omega)                  \right \|+\log\left\|A_2^a(x+k\omega)\cdot M^a_{l+k}\left(x\right)\left[\begin{array}{cc}
1 & 0\\
0 & 0
\end{array}\right]\right\|       ,\]
which means that we
decrease the length of $A^a_{m}$ by $k$ and increase the length of
$A^a_{1}$ by $k$. Adding these equations and dividing by $l$
yields
\[
\log\left\Vert \cM_{n}^{a}\left(x\right)\right\Vert +\sum_{j=l}^{\left(m-1\right)l-1}\frac{1}{l}\log\left\Vert M_{l}^{a}\left(x+j\omega\right)\right\Vert -\sum_{j=l}^{\left(m-2\right)l-1}\frac{1}{l}\log\left\Vert M_{2l}^{a}\left(x+j\omega\right)\right\Vert
-\sum_{k=0}^{l-1}\frac{1}{l}u_{k}\left(x\right)=O\left(\frac{1}{l}\right)
\]
up to a set of measure less than $3n\exp\left(- cL_l\left(\delta_0^l\right)^{-c_0}\right)$. For the functions $\displaystyle \sum_{j=l}^{\left(m-1\right)l-1}\frac{1}{l}\log\left\Vert M_{l}^{a}\left(x+j\omega\right)\right\Vert$ and $\displaystyle \sum_{j=l}^{\left(m-2\right)l-1}\frac{1}{l}\log\left\Vert M_{2l}^{a}\left(x+j\omega\right)\right\Vert$, Theorem \ref{sbet} can be applied. Note that $ml\sim n$. So,  the deviation $\delta$ is  the smallest deviation $\delta_0^n$ we can choose here. Then,
\[
\sum_{j=l}^{\left(m-1\right)l-1}\frac{1}{l}\log\left\Vert M_{l}^{a}\left(x+j\omega\right)\right\Vert -\sum_{j=l}^{\left(m-2\right)l-1}\frac{1}{l}\log\left\Vert M_{2l}^{a}\left(x+j\omega\right)\right\Vert
=\left(m-2\right)lL_{l}^{a}-\left(m-3\right)lL_{2l}^{a}+O\left(n\delta_0^n\right)
\]
up to a set of measure less than $\exp\left(-cn\delta_0^n\right)$. Note that $u_{k}$, $k=0,\ldots,l-1$  have the subharmonic extensions. Therefore, for any $\frac{u_k}{l}$, Theorem \ref{sbet} can be applied with $n=1$ and $\delta=\frac{n\delta_0^n}{l}$, and obtain that
\[
\sum_{k=0}^{l-1}\frac{1}{l}u_{k}\left(x\right)-\sum_{k=0}^{l-1}\frac{1}{l}\left\langle u_{k}\right\rangle =O\left(n\delta_0^n\right)
\]
up to a set of measure less than $l\exp(-cn^{1-c'}\delta_0^n)$. Thus, combining  these equations, we have that
\begin{equation}\label{M_impure_AP}
\log\left|f_{n}^{a}\left(x\right)\right|+\left(m-2\right)lL_{l}^{a}-\left(m-3\right)lL_{2l}^{a}-\sum_{k=0}^{l-1}\frac{1}{l}\left\langle u_{k}\right\rangle =O\left(n\delta_0^n\right)
\end{equation}
up to a set of measure less than $3n\exp\left(- cL_l\left(\delta_0^l\right)^{-c_0}\right)+l\exp(-cn^{1-c}\delta_0^n)+\exp\left(-cn\delta_0^n\right)$. Recalling that $\delta_0^n=
C_{\omega}n^{-\frac{1}{A}+}$, $c_0=\frac{1}{5}$ and $l\sim n^{c'}$, we have
\[3n\exp\left(- cL_l(E)\left(\delta_0^l\right)^{-c_0}\right)+l\exp(-cn^{1-c'}\delta_0^n)+\exp\left(-cn\delta_0^n\right)\le \exp\left(-c''n^{\frac{c'}{2}}\right) ,\]
where $c''$ is a small constant depending on $a,v,\omega$ and $E$. Integrating (\ref{M_impure_AP}) and using the fact that $\left\Vert \log\left|f_{n}^{a}\right|\right\Vert _{L^{2}\left(\TT\right)}\le Cn$, yields
\[\left<\log\left|f_{n}^{a}\left(x\right)\right|\right>+\left(m-2\right)lL_{l}^{a}-\left(m-3\right)lL_{2l}^{a}-\sum_{k=0}^{l-1}\frac{1}{l}\left\langle u_{k}\right\rangle =O\left(n\delta_0^n\right)+Cn\exp\left(-c''n^{\frac{c'}{2}}\right)=O\left(n\delta_0^n\right).\]
Combining it with (\ref{M_impure_AP}), we have
\begin{equation}
\left|\log\left|f_{n}^{a}\left(x\right)\right|-\left\langle \log\left|f_{n}^{a}\right|\right\rangle \right|=O\left(n\delta_0^n\right)\label{f^a-avg-bound}
\end{equation}
up to a set of measure less than $\exp\left(-c''n^{\frac{c'}{2}}\right)$.
Let $\cB$ be this exceptional set  and define
\[
\frac{1}{n}\log\left|f_{n}^{a}\right|-\left\langle \frac{1}{n}\log\left|f_{n}^{a}\right|\right\rangle =u_{0}+u_{1}
\]
 where $u_{0}=0$ on $\cB$ and $u_{1}=0$ on $\TT\setminus\cB$.
Obviously,  $\left\Vert u_{0}-\left\langle u_{0}\right\rangle \right\Vert _{L^{\infty}\left(\TT\right)}=O\left(\delta_0^n\right)$
and
\[
\left\Vert u_{1}\right\Vert _{L^{2}\left(\TT\right)}\le C\sqrt{\mes\left(\cB\right)}\le C\exp\left(-c''n^{\frac{c'}{2}}\right).
\]
Due to Lemma \ref{bgsbmonorm},
\[
\left\Vert \log\left|f_{n}^{a}\right|\right\Vert _{BMO\left(\TT\right)}=O\left(n\delta_0^n\right).
\]
\end{proof}
~\\

Similar to Lemma \ref{Ln_L}, we also can prove that $\langle \frac{1}{n}\log|f_n^a| \rangle $ in  Theorems \ref{ldtforfan}-\ref{ldth2} can be exchanged by $L^a$.
\begin{lemma}
\label{<f>-nL}There exists a constant $C_{0}=C_{0}\left(a,v,E,\omega,\gamma\right)$
such that
\[
\left|\left\langle \log\left|f_{n}^{a}\right|\right\rangle -nL_{n}^{a}\right|\le C_{0}
\]
for all integers.\end{lemma}
\begin{proof}
Recall that
\[
\log\left\Vert \cM_{n}^{a}\left(x\right)\right\Vert +\sum_{j=2}^{m-1}\log\left\Vert A_{j}^{a}\left(x\right)\right\Vert -\sum_{j=1}^{m-1}\log\left\Vert A_{j+1}^{a}\left(x\right)A_{j}^{a}\left(x\right)\right\Vert =O\left(\frac{1}{l}\right)
\]
up to a set of measure less than $3m\exp\left(- cL_l\left(\delta_0^l\right)^{-c_0}\right)$. Similarly,
\begin{eqnarray*}
 \log\left\Vert M_{n}^{a}\left(x\right)\right\Vert &&+\sum_{j=2}^{m-1}\log\left\Vert A_{j}^{a}\left(x\right)\right\Vert -\sum_{j=1}^{m-1}\log\left\Vert A_{j+1}^{a}\left(x\right)A_{j}^{a}\left(x\right)\right\Vert\\
 &&-\log\left\Vert M_{l'}^{a}\left(x+(m-1)l\omega\right)M_{l}^{a}\left(x+(m-2)l\omega\right)\right\Vert =O\left(\frac{1}{l}\right)
\end{eqnarray*}
up to a set of measure less than $3m\exp\left(-cL_l(E)l\right)$. Subtracting these two expressions
and then integrating, yields
\[
\left|\left\langle \log\left|f_{n}^{a}\right|\right\rangle -nL_{n}^{a}\right|\le CR\left(l\right)+O\left(\frac{1}{l}\right)
\]
where
\[
R\left(n\right)=\sup_{n/2\le m\le n}\left|\left\langle \log\left|f_{m}^{a}\right|\right\rangle -mL_{m}^{a}\right|,\ \mbox{and}\ \log n\ll l\ll n.
\]
Then, our conclusion is obtained by iterating this estimate.
\end{proof}

\section{The proof of Theorem \ref{thm:Number-of-eigenvalues}}
We used the LDTs and the Avalanche Principle together in  the above two proofs. As we have mentioned in the introduction, this method was first created in \cite{GS01} to prove the H\"older continuity of Lyapunov exponent $L^s(E,\omega)$ in $E$ with the strong Diophantine $\omega$. Recently, our second author also applied it to obtain the same continuity of $L(E,\omega)$ with any irrational $\omega$ in \cite{T18}. Our proof of Theorem  \ref{thm:Number-of-eigenvalues} needs this result. Therefore, we list it as a lemma:
\begin{lemma}\label{holder-contin}
 Assume $\beta(\omega)=0$ and $L(E_0,\omega)>0$. There exists $r_E=r_E(a,v,E_0,\omega)$ such that for any $|E-E_0|\leq r_E$,
 \[\frac{3}{4}L(E_0,\omega)<L(E,\omega)<\frac{5}{4}L(E_0,\omega).\]
 Furthermore, there exists a constant $h=h(a,v)$ called H\"older exponent such that for any $E_1,E_2\in [E_0-r_E,E_0+r_E]$,
 \begin{equation}\label{holder}
   \left|L(E_1,\omega)-L(E_2,\omega)\right|<|E_1-E_2|^h.
 \end{equation}
\end{lemma}

\begin{proof}[The proof of Theorem \ref{thm:Number-of-eigenvalues}]
From Theorem \ref{ldtforfan}-\ref{ldth2} and Lemma \ref{<f>-nL}, we have that for any $\delta>\delta_0^n$  and $\left(x,E\right)\in\TT\times\cE$
except for a set of measure $C\exp\left(-c\delta\left(\delta_0^n\right)^{-1}\right)$,
\begin{equation}
\left|\log\left|f_{n}^{a}\left(x,E,\omega\right)\right|-nL_{n}^{a}\left(E,\omega\right)\right|\le n\delta.\label{ldt_in_E}
\end{equation}
Then, due to Fubini's Theorem and Chebyshev's inequality, there exists a set $\cB_{n,\delta}\subset\TT$ with $\mes\cB_{n,\delta}<C\exp\left(-c\delta\left(\delta_0^n\right)^{-1}\right)$,
such that for each $x\in\TT\setminus\cB_{n,\delta}$ there exists
$\cE_{n,\delta,x}\subset\cD$, with $\mes\cE_{n,\delta,x}<C\exp\left(-c\delta\left(\delta_0^n\right)^{-1}\right)$ such that (\ref{ldt_in_E}) holds for any $E\in\cE\setminus\cE_{n,\delta,x}$. Therefore, there exist $x_{1},\, E_{1}$
satisfying
\[
\left|x_{1}-x_{0}\right|\le C\exp\left(-c\left(\delta_0^n\right)^{-\frac{1}{2}}\right),
\]and
\[
\left|E_{1}-E_{0}\right|\le C\exp\left(-c\left(\delta_0^n\right)^{-\frac{1}{2}}\right),
\]
such that
\begin{equation}
\log\left|f_{n}^{a}\left(x_{1},E_{1}\right)\right|\ge nL_{n}\left(E_{1}\right)-n\left(\delta_0^n\right)^{\frac{1}{2}}.\label{f^a-lb}
\end{equation}
Define
\[R:=\left(\delta_0^n\right)^{\frac{1}{h}}\gg C\exp\left(-c_{0}\left(\delta_0^n\right)^{-\frac{1}{2}}\right),\]
and
$$\cN_{x,E}\left(r\right)=\#\left\{ E:\, f_{n}^{a}\left(x,E'\right)=0,\,\left|E'-E\right|\le r\right\} .$$
The Jensen formula states that for any function $f$ analytic on a neighborhood of $\cD(z_0,R)$, see~\cite{levin},
\begin{equation}
 \label{eq:jensen}
 \int_0^1 \log |f(z_0+Re(\theta))|\, d\theta - \log|f(z_0)| = \sum_{\zeta:f(\zeta)=0} \log\frac{R}{|\zeta-z_0|}
\end{equation}
provided $f(z_0)\ne0$. Thus, we have that
\begin{equation}
\cN_{x_{1},E_{1}}\left(3R\right)\le \frac{1}{2\pi}\int_{0}^{2\pi}\log\left|f_{n}^{a}\left(x_{1},E_{1}+4Re^{i\theta}\right)\right|d\theta-\log\left|f_{n}^{a}\left(x_{1},E_{1}\right)\right|\label{Jensen-at-E1}.
\end{equation}
By Lemma \ref{M^a-upper-bound}, it
yields
\[
\cN_{x_{1},E_{1}}\left(3R\right)\le \left(\sup_{\left|E-E_{1}\right|=4R}\left(n\left(L_{n}^{a}\left(E\right)-L_{n}^{a}\left(E_{1}\right)\right)\right)\right)+3n\delta_0^n.
\]
Due to Lemma \ref{holder-contin}, if $|E_1-E_2|<4R$, then
\[\left|L(E_1)-L(E_2)\right|<\left|E_1-E_2\right|^{h}<4\delta_0^n.\]
Combining  it with Lemma \ref{Ln_L} and the fact that $\delta_0^n\gg\frac{\left(\log n\right)^2}{n}$, we have
\[\left(\sup_{\left|E-E_{1}\right|=4R}\left(n\left(L_n^{a}\left(E\right)-L_n^{a}\left(E_{1}\right)\right)\right)\right)<10 n\delta_0^n.\]
Thus,
\[\cN_{x_{1},E_{1}}\left(3R\right)\le 13 n\delta_0^n.\]
Recalling  that $|E_0-E_1|\ll R$, we have
\begin{equation}
\cN_{x_{1},E_{0}}\left(2R\right)\le\cN_{x_{1},E_{1}}\left(3R\right)\le 13n\delta_0^n.\label{z(R)-bound}
\end{equation}
Note that $H_n(x,\omega)$ is Hermitian. Thus,
by the Mean Value Theorem,
\[
\left\Vert H_n\left(x_{0},\omega\right)-H_n\left(x_{1},\omega\right)\right\Vert \le C\left|x_{0}-x_{1}\right|\le C\exp\left(-c_{0}\left(\delta_0^n\right)^{-\frac{1}{2}}\right).
\]
Let $E_{j}^{\left(n\right)}\left(x,\omega\right)$, $j=1,\ldots,n$ be the
eigenvalues of $H_n\left(x,\omega\right)$ ordered increasingly. Then,
\[
\left|E_{j}^{\left(n\right)}\left(x_{0}\right)-E_{j}^{\left(n\right)}\left(x_{1}\right)\right|\le C\exp\left(-c_{0}\left(\delta_0^n\right)^{-\frac{1}{2}}\right).
\]
This implies that
$$\cN_{x_{0},E_{0}}\left(R\right)\le\cN_{x_{1},E_{0}}\left(2R\right)<13n\delta_0^n.$$
\end{proof}
\begin{remark}
  Similarly,
  \[
\#\left\{ z\in\mathbb{C}:\, f_{n}^{a}\left(z,E_0,\omega\right)=0,\,\left|z-x_{0}\right|<\left(\delta_0^n\right)^{\frac{1}{h}}\right\} \le 13n\delta_0^n.
\]
\end{remark}

\end{document}